\documentclass[11pt,letterpaper]{amsart}

\usepackage[T1]{fontenc}
\usepackage{lmodern}

\usepackage{amsmath, amsthm, amssymb, amsfonts, mathtools}
\usepackage{mathrsfs}
\usepackage[mathcal]{eucal}
\usepackage{bm}
\usepackage{bbm}
\usepackage{comment}

\usepackage{graphicx,float}
\usepackage{tikz-cd}
\usepackage[all]{xy}
\usepackage{wasysym}
\usepackage{dsfont}

\usepackage{longtable,array}
\usepackage{enumitem}

\usepackage{cite}

\usepackage{tikz}
\usetikzlibrary{tikzmark}

\usepackage{hyperref}
\hypersetup{
    colorlinks=true,
    linkcolor=blue,
    filecolor=blue,
    urlcolor=blue,
    citecolor=red,
}

\numberwithin{equation}{section}

\theoremstyle{definition}

\theoremstyle{remark}

\numberwithin{equation}{section}

\theoremstyle{plain}

\begin{document}

\title[Transcendence of twisted special $L$-values]{On the transcendence of twisted special $L$-values at non-negative integers in characteristic $p$}

\author{Jing Ye}
\address{Department of Mathematics, Texas A\&M University,
	College Station, Texas, 77843, United States}
\email{yej@tamu.edu}
\date{\today}
\dedicatory{}

\commby{}

% ==============================================================================
% 1. THEOREM ENVIRONMENTS (定理环境配置)
% ==============================================================================
\theoremstyle{definition}
\newtheorem{Def}{Definition}[section]
\newtheorem{eg}[Def]{Example}
\newtheorem*{ex}{Exercise}
\newtheorem*{nota}{Notation}
\newtheorem{prob}[Def]{Problem}
\newtheorem{Q}[Def]{Question}
\newtheorem*{conv}{Convention}
\newtheorem{asp}[Def]{Assumption}

\theoremstyle{plain}
\newtheorem{thm}[Def]{Theorem}
\newtheorem{cor}[Def]{Corollary}
\newtheorem{prop}[Def]{Proposition}
\newtheorem{propdef}[Def]{Proposition-Definition}
\newtheorem{lem}[Def]{Lemma}
\newtheorem*{clm}{Claim}
\newtheorem{app}[Def]{Application}
\newtheorem{cjt}[Def]{Conjecture}
\newtheorem*{obs}{Observation}
\newtheorem{rev}{Review}

\theoremstyle{remark}
\newtheorem{rmk}[Def]{\normalfont\textit{Remark}}

\newenvironment{pf}{{\noindent\it Proof.}\quad}{\hfill $\qed$\par}
\renewcommand{\qedsymbol}{$\blacksquare$}

% ==============================================================================
% 2. STANDARD OPERATORS & FUNCTIONS (\operatorname 算子与函数)
% ==============================================================================
% --- Ring & Field Theory ---
\newcommand{\rad}[1]{\operatorname{rad}(#1)}
\newcommand{\supp}[1]{\operatorname{supp}(#1)}
\newcommand{\spec}[1]{\operatorname{Spec}(#1)}
\newcommand{\mspec}[1]{\operatorname{MaxSpec}(#1)}
\newcommand{\ann}[2]{\operatorname{ann}_{#1}(#2)}
\newcommand{\ass}[1]{\operatorname{Ass}(#1)}
\newcommand{\assb}[1]{\operatorname{Ass}\Big(#1\Big)}
\newcommand{\fra}[1]{\operatorname{Frac}(#1)}
\newcommand{\Frac}{\operatorname{Frac}}
\newcommand{\hgt}[1]{\textnormal{ht}(#1)}
\newcommand{\trdeg}{\textnormal{tr.deg}}

% --- Homology & Category ---
\newcommand{\im}{\operatorname{im}}
\newcommand{\img}{\operatorname{Im}}
\newcommand{\coim}{\operatorname{coim}}
\newcommand{\cok}{\operatorname{coker}}
\newcommand{\coker}{\mathrm{Coker}}
\newcommand{\image}{\mathrm{Im}}
\newcommand{\mor}[2]{\operatorname{Hom}({#1},{#2})}
\newcommand{\obj}[1]{\operatorname{obj}(#1)}
\newcommand{\homm}[3]{\operatorname{Hom}_{#1}({#2},{#3})}
\renewcommand{\hom}{\operatorname{Hom}}
\newcommand{\qhom}{\operatorname{Hom^\circ}}
\newcommand{\Ext}{\textnormal{Ext}}
\newcommand{\End}{\operatorname{End}}
\newcommand{\edo}{\mathrm{End}}
\newcommand{\en}[2]{\textnormal{End}_{#1}(#2)}
\newcommand{\aut}{\mathrm{Aut}}
\newcommand{\auto}[1]{\operatorname{Aut}(#1)}
\newcommand{\out}{\mathrm{Out}}
\newcommand{\Aut}{\mathcal{A}\mathrm{ut}}
\newcommand{\ob}{\operatorname{Ob}}
\newcommand{\ev}{\mathrm{ev}}
\newcommand{\coev}{\mathrm{coev}}
\newcommand{\funend}{\EuScript{E}\mathrm{nd}}
\newcommand{\rep}{\mathrm{Rep}}
\newcommand{\Rep}{\mathbf{Rep}}
\newcommand{\fun}{\mathrm{Fun}}
\newcommand{\Fun}{\EuScript{F}\mathrm{un}}
\newcommand{\colim}{\varinjlim}
\newcommand{\Lim}{\varprojlim}
\newcommand{\Lan}{\mathrm{Lan}}
\newcommand{\Ran} {\mathrm{Ran}}

% --- Linear Algebra & Matrices ---
\newcommand{\Tr}{\operatorname{Tr}}
\newcommand{\Lie}{\operatorname{Lie}}
\newcommand{\rk}{\operatorname{rank}}
\newcommand{\Fitt}{\operatorname{Fitt}}
\newcommand{\fitt}{\operatorname{Fitt}}
\newcommand{\Mat}{\operatorname{Mat}}
\newcommand{\SL}{\operatorname{SL}}
\newcommand{\GL}{\operatorname{GL}}
\newcommand{\AGL}{\mathbb{A}\GL}
\newcommand{\Span}{\operatorname{Span}}

% --- Arithmetic & Drinfeld/Anderson Modules ---
\newcommand{\Div}{\operatorname{Div}}
\newcommand{\Prin}{\operatorname{Prin}}
\newcommand{\Char}{\operatorname{char}}
\newcommand{\ord}{\operatorname{ord}}
\newcommand{\Reg}{\operatorname{Reg}}
\newcommand{\gal}{\operatorname{Gal}}
\newcommand{\Gal}{\operatorname{Gal}}
\newcommand{\Frob}{\operatorname{Frob}}
\newcommand{\Pic}{\operatorname{Pic}}
\newcommand{\Res}{\operatorname{Res}}
\newcommand{\Cl}{\operatorname{Cl}}
\newcommand{\inv}[1]{\operatorname{Inv}(#1)}
\newcommand{\disc}[1]{\textnormal{disc}({#1})}
\newcommand{\Ad}{\operatorname{Ad}}
\newcommand{\sgn}{\operatorname{sgn}}
\renewcommand{\div}{\operatorname{div}}
\newcommand{\der}{\operatorname{Der}}

% --- Miscellaneous Operator Modifiers ---
\newcommand{\spl}{\textnormal{spec}}
\newcommand{\sat}{\textnormal{sat}}
\newcommand{\ind}{\textnormal{ind}}
\newcommand{\Ind}{\operatorname{Ind}}
\newcommand{\op}{\operatorname{op}}
\newcommand{\eff}{\textnormal{eff}}
\newcommand{\rig}{\textnormal{rig}}
\newcommand{\st}{\textnormal{St}}
\newcommand{\codim}{\operatorname{codim}}
\newcommand{\cov}{\operatorname{Cov}}
\newcommand{\cat}{\operatorname{Cat}}
\newcommand{\Max}{\operatorname{Max}}
\renewcommand{\sp}{\operatorname{Sp}}
\newcommand{\gr}[2]{\textnormal{gr}_{#1}(#2)}
\newcommand{\grd}[3]{\textnormal{gr}_{#1}^{#2}(#3)}
\newcommand{\br}{\operatorname{Br}}
\newcommand{\ab}{\operatorname{ab}}
\newcommand{\tdiv}{\operatorname{\mid\!\mid}}
\newcommand{\id}{\mathrm{id}}

% ==============================================================================
% 3. ALPHABETIC MATH FONTS (分类补全的字母表)
% ==============================================================================

% --- BB Font (\mathbb) ---
\renewcommand{\AA}{\mathbb{A}}
\newcommand{\BB}{\mathbb{B}}
\newcommand{\CC}{\mathbb{C}}
\newcommand{\DD}{\mathbb{D}}
\newcommand{\EE}{\mathbb{E}}
\newcommand{\FF}{\mathbb{F}}
\newcommand{\GG}{\mathbb{G}}
\newcommand{\HH}{\mathbb{H}}
\newcommand{\II}{\mathbb{I}}
\newcommand{\JJ}{\mathbb{J}}
\newcommand{\KK}{\mathbb{K}}
\newcommand{\LL}{\mathbb{L}}
\newcommand{\MM}{\mathbb{M}}
\newcommand{\NN}{\mathbb{N}}
\newcommand{\OO}{\mathbb{O}}
\newcommand{\PP}{\mathbb{P}}
\newcommand{\QQ}{\mathbb{Q}}
\newcommand{\RR}{\mathbb{R}}
\renewcommand{\SS}{\mathbb{S}}
\newcommand{\TT}{\mathbb{T}}
\newcommand{\UU}{\mathbb{U}}
\newcommand{\VV}{\mathbb{V}}
\newcommand{\WW}{\mathbb{W}}
\newcommand{\XX}{\mathbb{X}}
\newcommand{\YY}{\mathbb{Y}}
\newcommand{\ZZ}{\mathbb{Z}}

% --- Cal Font (\mathcal) ---
\newcommand{\mcA}{\mathcal{A}}
\newcommand{\mcB}{\mathcal{B}}
\newcommand{\mcC}{\mathcal{C}}
\newcommand{\mcD}{\mathcal{D}}
\newcommand{\mcE}{\mathcal{E}}
\newcommand{\mcF}{\mathcal{F}}
\newcommand{\mcG}{\mathcal{G}}
\newcommand{\mcH}{\mathcal{H}}
\newcommand{\mcI}{\mathcal{I}}
\newcommand{\mcJ}{\mathcal{J}}
\newcommand{\mcK}{\mathcal{K}}
\newcommand{\mcL}{\mathcal{L}}
\newcommand{\mcM}{\mathcal{M}}
\newcommand{\mcN}{\mathcal{N}}
\newcommand{\mcO}{\mathcal{O}}
\newcommand{\mcP}{\mathcal{P}}
\newcommand{\mcQ}{\mathcal{Q}}
\newcommand{\mcR}{\mathcal{R}}
\newcommand{\mcS}{\mathcal{S}}
\newcommand{\mcT}{\mathcal{T}}
\newcommand{\mcU}{\mathcal{U}}
\newcommand{\mcV}{\mathcal{V}}
\newcommand{\mcW}{\mathcal{W}}
\newcommand{\mcX}{\mathcal{X}}
\newcommand{\mcY}{\mathcal{Y}}
\newcommand{\mcZ}{\mathcal{Z}}

% --- Script Font (\mathscr) ---
\newcommand{\sA}{\mathscr{A}}
\newcommand{\sB}{\mathscr{B}}
\newcommand{\sC}{\mathscr{C}}
\newcommand{\sD}{\mathscr{D}}
\newcommand{\sE}{\mathscr{E}}
\newcommand{\sF}{\mathscr{F}}
\newcommand{\sG}{\mathscr{G}}
\newcommand{\sH}{\mathscr{H}}
\newcommand{\sI}{\mathscr{I}}
\newcommand{\sJ}{\mathscr{J}}
\newcommand{\sK}{\mathscr{K}}
\newcommand{\sL}{\mathscr{L}}
\newcommand{\sM}{\mathscr{M}}
\newcommand{\sN}{\mathscr{N}}
\newcommand{\sO}{\mathscr{O}}
\newcommand{\sP}{\mathscr{P}}
\newcommand{\sQ}{\mathscr{Q}}
\newcommand{\sR}{\mathscr{R}}
\newcommand{\sS}{\mathscr{S}}
\newcommand{\sT}{\mathscr{T}}
\newcommand{\sU}{\mathscr{U}}
\newcommand{\sV}{\mathscr{V}}
\newcommand{\sW}{\mathscr{W}}
\newcommand{\sX}{\mathscr{X}}
\newcommand{\sY}{\mathscr{Y}}
\newcommand{\sZ}{\mathscr{Z}}

% --- EuScript Font (\EuScript) ---
\newcommand{\CA}{\EuScript{A}}
\newcommand{\CB}{\EuScript{B}}
\newcommand{\CCC}{\EuScript{C}}
\newcommand{\CDD}{\EuScript{D}}
\newcommand{\CE}{\EuScript{E}}
\newcommand{\CF}{\EuScript{F}}
\newcommand{\CG}{\EuScript{G}}
\newcommand{\CH}{\EuScript{H}}
\newcommand{\CI}{\EuScript{I}}
\newcommand{\CJ}{\EuScript{J}}
\newcommand{\CK}{\EuScript{K}}
\newcommand{\CL}{\EuScript{L}}
\newcommand{\CM}{\EuScript{M}}
\newcommand{\CN}{\EuScript{N}}
\newcommand{\CO}{\EuScript{O}}
\newcommand{\CP}{\EuScript{P}}
\newcommand{\CQ}{\EuScript{Q}}
\newcommand{\CR}{\EuScript{R}}
\newcommand{\CS}{\EuScript{S}}
\newcommand{\CT}{\EuScript{T}}
\newcommand{\CU}{\EuScript{U}}
\newcommand{\CV}{\EuScript{V}}
\newcommand{\CW}{\EuScript{W}}
\newcommand{\CX}{\EuScript{X}}
\newcommand{\CY}{\EuScript{Y}}
\newcommand{\CZ}{\EuScript{Z}}

% --- Fraktur Font (\mathfrak) ---
\newcommand{\mfa}{\mathfrak{a}} \newcommand{\mfA}{\mathfrak{A}} \newcommand{\fA}{\mathfrak{A}}
\newcommand{\mfb}{\mathfrak{b}} \newcommand{\mfB}{\mathfrak{B}} \newcommand{\fB}{\mathfrak{B}}
\newcommand{\mfc}{\mathfrak{c}} \newcommand{\mfC}{\mathfrak{C}}
\newcommand{\mfd}{\mathfrak{d}} \newcommand{\mfD}{\mathfrak{D}}
\newcommand{\mfe}{\mathfrak{e}} \newcommand{\mfE}{\mathfrak{E}}
\newcommand{\mff}{\mathfrak{f}} \ \newcommand{\mfF}{\mathfrak{F}}
\newcommand{\mfg}{\mathfrak{g}} \newcommand{\mfG}{\mathfrak{G}}
\newcommand{\mfh}{\mathfrak{h}} \newcommand{\mfH}{\mathfrak{H}}
\newcommand{\mfi}{\mathfrak{i}} \newcommand{\mfI}{\mathfrak{I}}
\newcommand{\mfj}{\mathfrak{j}} \newcommand{\mfJ}{\mathfrak{J}}
\newcommand{\mfk}{\mathfrak{k}} \newcommand{\mfK}{\mathfrak{K}}
\newcommand{\mfl}{\mathfrak{l}} \ \newcommand{\mfL}{\mathfrak{L}}
\newcommand{\mfm}{\mathfrak{m}} \newcommand{\mfM}{\mathfrak{M}}
\newcommand{\mfn}{\mathfrak{n}} \newcommand{\mfN}{\mathfrak{N}}
\newcommand{\mfo}{\mathfrak{o}} \newcommand{\mfO}{\mathfrak{O}}
\newcommand{\mfp}{\mathfrak{p}} \newcommand{\mfP}{\mathfrak{P}}
\newcommand{\mfq}{\mathfrak{q}} \newcommand{\mfQ}{\mathfrak{Q}}
\newcommand{\mfr}{\mathfrak{r}} \newcommand{\mfR}{\mathfrak{R}}
\newcommand{\mfs}{\mathfrak{s}} \newcommand{\mfS}{\mathfrak{S}}
\newcommand{\mft}{\mathfrak{t}} \ \newcommand{\mfT}{\mathfrak{T}}
\newcommand{\mfu}{\mathfrak{u}} \newcommand{\mfU}{\mathfrak{U}}
\newcommand{\mfv}{\mathfrak{v}} \newcommand{\mfV}{\mathfrak{V}}
\newcommand{\mfw}{\mathfrak{w}} \newcommand{\mfW}{\mathfrak{W}}
\newcommand{\mfx}{\mathfrak{x}} \newcommand{\mfX}{\mathfrak{X}}
\newcommand{\mfy}{\mathfrak{y}} \newcommand{\mfY}{\mathfrak{Y}}
\newcommand{\mfz}{\mathfrak{z}} \newcommand{\mfZ}{\mathfrak{Z}}

% --- Bold Font (\mathbf) ---
\newcommand{\bfa}{\mathbf{A}} \newcommand{\bA}{\mathbf{A}}
\newcommand{\bfb}{\mathbf{B}} \newcommand{\bB}{\mathbf{B}}
\newcommand{\bfc}{\mathbf{C}} \newcommand{\bC}{\mathbf{C}}
\newcommand{\bfd}{\mathbf{D}} \newcommand{\bD}{\mathbf{D}}
\newcommand{\bfe}{\mathbf{E}} \newcommand{\bE}{\mathbf{E}}
\newcommand{\bff}{\mathbf{F}} \newcommand{\bF}{\mathbf{F}}
\newcommand{\bfg}{\mathbf{G}} \newcommand{\bG}{\mathbf{G}}
\newcommand{\bfh}{\mathbf{H}} \newcommand{\bH}{\mathbf{H}}
\newcommand{\bfi}{\mathbf{i}} \newcommand{\bI}{\mathbf{I}}
\newcommand{\bfj}{\mathbf{j}} \newcommand{\bJ}{\mathbf{J}}
\newcommand{\bfk}{\mathbf{K}} \newcommand{\bK}{\mathbf{K}}
\newcommand{\bfl}{\mathbf{L}} \newcommand{\bL}{\mathbf{L}}
\newcommand{\bfm}{\mathbf{m}} \newcommand{\bM}{\mathbf{M}}
\newcommand{\bfn}{\mathbf{n}} \newcommand{\bN}{\mathbf{N}}
\newcommand{\bfo}{\mathbf{O}} \newcommand{\bO}{\mathbf{O}}
\newcommand{\bfp}{\mathbf{P}} \newcommand{\bP}{\mathbf{P}}
\newcommand{\bfq}{\mathbf{Q}} \newcommand{\bQ}{\mathbf{Q}}
\newcommand{\bfr}{\mathbf{R}} \newcommand{\bR}{\mathbf{R}}
\newcommand{\bfs}{\mathbf{S}} \newcommand{\bS}{\mathbf{S}}
\newcommand{\bft}{\mathbf{T}} \newcommand{\bT}{\mathbf{T}}
\newcommand{\bfu}{\mathbf{u}} \newcommand{\bU}{\mathbf{U}}
\newcommand{\bfv}{\mathbf{v}} \newcommand{\bV}{\mathbf{V}}
\newcommand{\bfw}{\mathbf{w}} \newcommand{\bW}{\mathbf{W}}
\newcommand{\bfx}{\mathbf{X}} \newcommand{\bX}{\mathbf{X}}
\newcommand{\bfy}{\mathbf{Y}} \newcommand{\bY}{\mathbf{Y}}
\newcommand{\bfz}{\mathbf{Z}} \newcommand{\bZ}{\mathbf{Z}}

% --- Sans Serif Font (\mathsf) ---
\newcommand{\sfA}{\mathsf{A}}
\newcommand{\sfB}{\mathsf{B}}
\newcommand{\sfC}{\mathsf{C}}
\newcommand{\sfD}{\mathsf{D}}
\newcommand{\sfE}{\mathsf{E}}
\newcommand{\sfF}{\mathsf{F}}
\newcommand{\sfG}{\mathsf{G}}
\newcommand{\sfH}{\mathsf{H}}
\newcommand{\sfI}{\mathsf{I}}
\newcommand{\sfJ}{\mathsf{J}}
\newcommand{\sfK}{\mathsf{K}}
\newcommand{\sfL}{\mathsf{L}}
\newcommand{\sfM}{\mathsf{M}}
\newcommand{\sfN}{\mathsf{N}}
\newcommand{\sfO}{\mathsf{O}}
\newcommand{\sfP}{\mathsf{P}}
\newcommand{\sfQ}{\mathsf{Q}}
\newcommand{\sfR}{\mathsf{R}}
\newcommand{\sfS}{\mathsf{S}}
\newcommand{\sfT}{\mathsf{T}}
\newcommand{\sfU}{\mathsf{U}}
\newcommand{\sfV}{\mathsf{V}}
\newcommand{\sfW}{\mathsf{W}}
\newcommand{\sfX}{\mathsf{X}}
\newcommand{\sfY}{\mathsf{Y}}
\newcommand{\sfZ}{\mathsf{Z}}

% ==============================================================================
% 4. VECTORS, GREEK & SPECIAL ELEMENTS (\boldsymbol / Overlines / Tildes)
% ==============================================================================

% --- Bold Italic Vectors (\boldsymbol) ---
\newcommand{\ba}{\boldsymbol{a}} 
\newcommand{\bb}{\boldsymbol{b}}
\newcommand{\bc}{\boldsymbol{c}} 
\newcommand{\bd}{\boldsymbol{d}}
\newcommand{\be}{\boldsymbol{e}} 
\newcommand{\bsf}{\boldsymbol{f}}
\newcommand{\bg}{\boldsymbol{g}} 
\newcommand{\bh}{\boldsymbol{h}}
\newcommand{\bi}{\boldsymbol{i}} 
\newcommand{\bj}{\boldsymbol{j}}
\newcommand{\bk}{\boldsymbol{k}} 
\newcommand{\bl}{\boldsymbol{l}}
\newcommand{\bsm}{\boldsymbol{m}} 
\newcommand{\bn}{\boldsymbol{n}}
\newcommand{\bo}{\boldsymbol{o}} 
\newcommand{\bp}{\boldsymbol{p}}
\newcommand{\bq}{\boldsymbol{q}} 
\newcommand{\bsr}{\boldsymbol{r}}
\newcommand{\bs}{\boldsymbol{s}}   
\newcommand{\bt}{\boldsymbol{t}}
\newcommand{\bu}{\boldsymbol{u}}   
\newcommand{\bv}{\boldsymbol{v}}
\newcommand{\bw}{\boldsymbol{w}} 
\newcommand{\bx}{\boldsymbol{x}}
\newcommand{\by}{\boldsymbol{y}}   
\newcommand{\bz}{\boldsymbol{z}}

% --- Bold Greek Symbols ---
\newcommand{\balpha}{{\bm{\alpha}}}
\newcommand{\bbeta}{\boldsymbol{\beta}}
\newcommand{\bdelta}{\bm{\delta}}
\newcommand{\blambda}{\boldsymbol{\lambda}}
\newcommand{\bomega}{\boldsymbol{\omega}}
\newcommand{\bmu}{\boldsymbol{\mu}}
\newcommand{\bnu}{\boldsymbol{\nu}}
\newcommand{\bpi}{\boldsymbol{\pi}}

% --- Overlines (\overline) ---
\newcommand{\oD}{\mkern2.5mu\overline{\mkern-2.5mu D}}
\newcommand{\oE}{\mkern2.5mu\overline{\mkern-2.5mu E}}
\newcommand{\oF}{\mkern2.5mu\overline{\mkern-2.5mu F}}
\newcommand{\oK}{\mkern2.5mu\overline{\mkern-2.5mu K}}
\newcommand{\oL}{\overline{L}}
\newcommand{\oM}{\mkern2.5mu\overline{\mkern-2.5mu M}}
\newcommand{\oP}{\mkern2.5mu\overline{\mkern-2.5mu P}}
\newcommand{\oQ}{\overline{Q}}
\newcommand{\oR}{\mkern2.5mu\overline{\mkern-2.5mu R}}
\newcommand{\oS}{\mkern2.5mu\overline{\mkern-2.5mu S}}
\newcommand{\oT}{\overline{T}}
\newcommand{\oU}{\overline{U}}
\newcommand{\oW}{\overline{W}}
\newcommand{\oX}{\overline{X}}
\newcommand{\oZ}{\mkern2.5mu\overline{\mkern-2.5mu Z}}
\newcommand{\og}{\overline{g}}
\newcommand{\oh}{\overline{h}}
\newcommand{\ok}{\overline{k}}
\newcommand{\on}{\overline{n}}
\newcommand{\ox}{\overline{x}}
\newcommand{\oalpha}{\overline{\alpha}}
\newcommand{\obeta}{\overline{\beta}}
\newcommand{\oGamma}{\overline{\Gamma}}
\newcommand{\ochi}{\overline{\chi}}
\newcommand{\oeta}{\overline{\eta}}
\newcommand{\okappa}{\overline{\kappa}}
\newcommand{\ophi}{\mkern2.5mu\overline{\mkern-2.5mu \phi}}
\newcommand{\opsi}{\mkern2.5mu\overline{\mkern-2.5mu \psi}}
\newcommand{\otheta}{\mkern2.5mu\overline{\mkern-2.5mu \theta}}
\newcommand{\oPhi}{\overline{\Phi}}
\newcommand{\oTheta}{\overline{\Theta}}
\newcommand{\oUpsilon}{\overline{\Upsilon}}
\newcommand{\oxi}{\overline{\xi}}
\newcommand{\oXi}{\overline{\Xi}}
\newcommand{\ozeta}{\overline{\zeta}}
\newcommand{\obn}{\overline{\bn}}
\newcommand{\Qbar}{\overline{\QQ}}
\newcommand{\oFq}{\overline{\FF}_q}
\newcommand{\oFqt}{\overline{\FF_q(t)}}
\newcommand{\oEE}{\overline{\EE}}
\newcommand{\soEE}{\mkern1mu\overline{\mkern-1mu \EE}}
\newcommand{\obE}{\overline{\bE}}
\newcommand{\sobE}{\mkern1mu\overline{\mkern-1mu \bE}}
\newcommand{\oFF}{\overline{\FF}}
\newcommand{\oinfty}{\overline{\infty}}

% --- Tildes (\widetilde) ---
\newcommand{\tB}{\widetilde{B}}
\newcommand{\tC}{\widetilde{C}}
\newcommand{\tE}{\widetilde{E}}
\newcommand{\tcE}{\widetilde{\cE}}
\newcommand{\tM}{\widetilde{M}}
\newcommand{\tP}{\widetilde{P}}
\newcommand{\tQ}{\widetilde{Q}}
\newcommand{\tS}{\widetilde{S}}
\newcommand{\tT}{\widetilde{T}}
\newcommand{\tU}{\widetilde{U}}
\newcommand{\tV}{\widetilde{V}}
\newcommand{\tX}{\widetilde{X}}
\newcommand{\teps}{\widetilde{\varepsilon}}
\newcommand{\tbe}{\widetilde{\be}}
\newcommand{\tg}{\widetilde{g}}
\newcommand{\tilh}{\widetilde{h}}
\newcommand{\tiota}{\widetilde{\iota}}
\newcommand{\tfp}{\widetilde{\fp}}
\newcommand{\tfq}{\widetilde{\fq}}
\newcommand{\tpi}{\widetilde{\pi}}
\newcommand{\tphi}{\widetilde{\phi}}
\newcommand{\tPhi}{\widetilde{\Phi}}
\newcommand{\tPsi}{\widetilde{\Psi}}
\newcommand{\trho}{\widetilde{\rho}}
\newcommand{\ttheta}{\widetilde{\theta}}
\newcommand{\tx}{\widetilde{x}}
\newcommand{\ty}{\widetilde{y}}

% --- Other Hat / Underline / Base Modifiers ---
\newcommand{\udl}[1]{\underline{#1}}
\newcommand{\ui}{\underline{i}}
\newcommand{\ufp}{\underline{\fp}}
\newcommand{\ufq}{\underline{\fq}}
\newcommand{\ut}{\underline{t}}
\newcommand{\htheta}{\hat{\theta}}
\newcommand{\hzeta}{\hat{\zeta}}
\newcommand{\ulM}{\underline{M}}
\newcommand{\btau}{\bar{\tau}}
\newcommand{\Fqts}{\FF_q[\ut_s]}

% ==============================================================================
% 5. CATEGORY LABELS & STRINGS (\textnormal / \mathbf Category Names)
% ==============================================================================
\newcommand{\comr}{\textnormal{\textbf{ComRings}}}
\newcommand{\Mod}{\textnormal{\textbf{Mod}}}
\newcommand{\lmod}{\sideset{_R}{}{\mathop{\Mod}}}
\newcommand{\rmod}{\sideset{}{_R}{\mathop{\Mod}}}
\newcommand{\mmod}[1]{\sideset{_{#1}}{}{\mathop{\Mod}}}
\newcommand{\LMod}{\mathrm{LMod}}
\newcommand{\RMod}{\mathrm{RMod}}
\newcommand{\BMod}{\mathrm{BMod}}
\newcommand{\set}{\textnormal{\textbf{Sets}}}
\newcommand{\Set}{\EuScript{S}\mathrm{et}}
\newcommand{\sSet}{\mathbf{sSet}}
\newcommand{\grp}{\textnormal{\textbf{Grp}}}
\newcommand{\Ab}{\textnormal{{\textbf{Ab}}}}
\newcommand{\Top}{\textnormal{\textbf{Top}}}
\newcommand{\cTop}{\mathrm{Top}}
\newcommand{\Topo}{\textnormal{\textbf{Top}}}
\newcommand{\Ring}{\EuScript{R}\mathrm{ing}}
\newcommand{\cring}{\EuScript{C}\mathrm{ring}}
\newcommand{\abel}{\EuScript{A}\mathrm{bel}}
\newcommand{\vect}{\mathrm{Vect}}
\newcommand{\hilb}{\mathrm{Hilb}}
\newcommand{\andm}{\textnormal{-}\textsf{AndMod}}
\newcommand{\abm}{\textnormal{-}\textsf{AbMod}}
\newcommand{\mot}{\textnormal{-}\textsf{Mot}}
\newcommand{\motI}{\textnormal{-}\textsf{MotI}}
\newcommand{\effmot}{\textnormal{-}\textsf{Mot}^{\eff}}
\newcommand{\fgeffmot}{\textnormal{-}\textsf{Mot}^{\eff}_{<\infty}}
\newcommand{\imot}{\textnormal{-}\textsf{Mot}^{\circ}}
\newcommand{\armot}{\textnormal{-}\textsf{Mot}_{\textnormal{Artin}}^{\circ}}
\newcommand{\reps}{\textnormal{-}\textsf{Reps}}

% ==============================================================================
% 6. RELATIONS, ARROWS & STRUCTURAL COMMANDS (关系符、箭头及简写环境)
% ==============================================================================
\newcommand{\assign}{\mathrel{\vcenter{\baselineskip0.5ex \lineskiplimit0pt
			\hbox{\scriptsize.}{\hbox{\scriptsize.}}}}%
	=}
\newcommand{\rassign}{=%
	\mathrel{\vcenter{\baselineskip0.5ex \lineskiplimit0pt
			\hbox{\scriptsize.}{\hbox{\scriptsize.}}}}%
}
\newcommand{\isoto}{\stackrel{\sim}{\to}}
\newcommand{\iso}{\stackrel{\sim}{\longrightarrow}}
\newcommand{\upto}[1]{\overset{#1}{\to}}
\newcommand{\uparr}[1]{\stackrel{\to}{#1}}
\newcommand{\inlim}{\varprojlim}
\newcommand{\dlim}{\varinjlim}
\newcommand{\plim}{\varprojlim}
\newcommand{\dmn}{\trianglerighteq}
\newcommand{\mayeq}{\stackrel{?}{=}}

% --- Environments Shortcuts ---
\newcommand\beq{\begin{equation}}
	\newcommand\eeq{\end{equation}}
\newcommand\bea{\begin{eqnarray}}
	\newcommand\eea{\end{eqnarray}}
\newcommand\nn{\nonumber \\}
\newcommand{\rB}{\textnormal{B}}
% ==============================================================================
% 7. MISCELLANEOUS CUSTOM SYMBOLS & MATH SHAPES
% ==============================================================================
\definecolor{dummycolor}{rgb}{0,0,0} % 占位环境符，如需要可配合其余包
\newcommand{\z}[1]{\mathbb{Z}/#1\mathbb{Z}}
\newcommand{\dis}{\displaystyle}
\newcommand{\eps}{\varepsilon}
\newcommand{\one}{\mathbf{1}}
\newcommand{\bbk}{\mathbb{k}}
\newcommand{\forget}{\mathbf{f}}
\newcommand{\sep}{\textnormal{sep}}
\newcommand{\inn}{\mathrm{in}}
\newcommand{\si}{\mathrm{si}}
\newcommand{\sr}{\mathrm{sr}}
\newcommand{\perf}{\mathrm{perf}}
\newcommand{\tors}{\mathrm{tors}}
\newcommand{\nr}{\mathrm{nr}}
\newcommand{\Ga}{\GG_{\mathrm{a}}}
\newcommand{\Gm}{\GG_{\mathrm{m}}}
\newcommand{\LLhat}{\widehat{\LL}}
\definecolor{C}{rgb}{0,0,0} \newcommand{\C}{\CC_{\infty}}
\newcommand{\tauid}{{\tau=\mathrm{id}}}
\newcommand{\sigmaid}{{\sigma=\mathrm{id}}}
\newcommand{\power}[2]{{#1 [[ #2 ]]}}
\newcommand{\laurent}[2]{{#1 (( #2 ))}}
\newcommand{\brac}[2]{\genfrac{\{}{\}}{0pt}{}{#1}{#2}}
\newcommand{\abs}[1]{\left|#1\right|}
\newcommand{\pphi}{\varphi}
\newcommand{\ksym}[2]{\left(\frac{#1}{#2}\right)}
\newcommand{\res}[3]{\left(\frac{#1}{#2}\right)_{#3}}
\newcommand{\resi}[3]{\left(\frac{#1}{#2}\right)_{#3}}

% --- Norm Expressions ---
\newcommand{\norm}[1]{|\!|#1|\!|}
\newcommand{\bignorm}[1]{\bigg|\!\bigg|#1\bigg|\!\bigg|}
\newcommand{\dnorm}[1]{\lVert #1 \rVert}
\newcommand{\cdnorm}[1]{\lVert #1 \rVert}
\newcommand{\inorm}[1]{{\lvert #1 \rvert}_{\infty}}
\newcommand{\idnorm}[1]{{\lVert #1 \rVert}_{\infty}}
\newcommand{\diam}[1]{\langle #1 \rangle}
\newcommand{\smod}[1]{{\, (\mathrm{mod}\, #1)}}
\newcommand{\Aord}[2]{{[ #1 ]}_{#2}}
\newcommand{\bigAord}[2]{{\left[ #1 \right]}_{#2}}

% --- Extra Matrix/Decoration Tools ---
\newcommand{\tr}{{\mathsf{T}}}
\newcommand{\bdot}{\mathbin{.}}
\newcommand\braceover[2]{%
	\makebox[-2pt][l]{$\smash{\overbrace{\phantom{%
					\begin{bmatrix}#2\end{bmatrix}}}^{\text{#1}}}$}#2}

\newcommand{\conj}[1]{\overline{#1}}
\newcommand{\Spec}{\operatorname{Spec}}
\renewcommand{\laurent}[2]{#1(\!(#2)\!)}

\begin{abstract}
In this article, we study the transcendence of special values of certain Goss type $L$-series at non-negative integers, which takes values in a function field of characteristic $p$. We show that, for a Drinfeld module $\varphi$ over $K$ and an Artin representation $\rho:G_K\to \GL_n(\conj \FF_q)$, the twisted special value $L(\varphi^\vee,\rho,k)$ is transcendental over $K$ for every non-negative integer $k$. The proof uses the theory of Artin twists of Drinfeld modules, Taelman's regulators of $t$-modules, and the algebraic independence theorem of Gezmi{\c s} and Namoijam for tractable coordinates of logarithms. As a consequence, we deduce the transcendence of the special $L$-value $L(\rho,k)$ for every positive integers $k$.
\end{abstract}

\maketitle

\section{Introduction}
The arithmetic of special values of Goss $L$-series is one of the central themes in function field arithmetic; see \cite{Goss}. In the motivic language of Anderson and Drinfeld, these $L$-values are attached to $t$-motives and their associated $t$-modules. A common theme in Taelman's work and its subsequent developments is the relation between special $L$-values at $s=0$ and regulators of $t$-modules. In \cite{taelman2009special}, Taelman conjectured that, for a uniformizable abelian $t$-module with everywhere good reduction, there exists a submodule of suitable rank in its Lie algebra whose exponential consists of integral points and its special $L$-value at $s=0$ is governed by a relevant ratio of covolumes. Taelman's class number formula for Drinfeld modules \cite{taelman2012special} can be viewed as a special case of this conjecture. Fang subsequently established a class formula for abelian $t$-modules in \cite{fang2015special}. Angl{\`e}s, Ngo Dac, and Tavares Ribeiro later proved Taelman's conjecture for a large class of $t$-modules in \cite{ANDTR20} and established a class formula for admissible Anderson modules in \cite{angles2022class}. The resulting regulators can be represented as determinants of matrices whose entries are logarithms of algebraic points. On the transcendence side, Papanikolas's Tannakian theory \cite{Pap08} and the theorem of Chang--Papanikolas on periods and logarithms of Drinfeld modules \cite{CP12} provide a powerful algebraic independence criterion. 

More recently, Gezmi{\c s} and Namoijam introduced the Anderson $t$-modules $$G_k=\varphi\otimes C^{\otimes k}$$ where $C$ is the Carlitz module, and proved some algebraic independence results for the tractable coordinates of their logarithms \cite[Theorem 1.3]{GN25}. The notion of tractable coordinates, due to Brownawell and Papanikolas, and given precise formulations in \cite[Definition 3.3.1]{CM21} and \cite[Definition 5.14]{CCM22}, abstracts the last-coordinate logarithm phenomenon for tensor powers of the Carlitz module studied by Anderson--Thakur \cite{AT90} and Jing Yu \cite{Yu91}. Gezmi{\c s} and Namoijam further applied these results and obtained a transcendence result for the special values $L(M_\varphi,k)$ at non-negative integers $k$ in \cite[Theorem 1.1]{GN24}, where $M_\varphi$ is the $t$-motive attached to $\varphi$.
For a Drinfeld module $\varphi_t=\theta+a_1\tau$ of rank 1, this
construction gives the example $G_k=C_{k+1}^{(a_1)}$,  where $C_{k+1}^{(a_1)}$ is the generalized Carlitz $(k+1)$-st tensor power $t$-module whose $\tau$-coefficient matrix has the single lower-left entry $a_1$.

 Throughout this article, we let $\bA=\FF_q[t]$, $A=\FF_q[\theta]$, and $K=\Frac(A)=\FF_q(\theta)$. We fix the structure map $\iota:\bA\longrightarrow K$, given by $t\mapsto \theta$, which makes $K$ into an $\bA$-field. We let $K_\infty = \FF_q(\!(1/\theta)\!)$ denote the completion of $K$ at the infinite place $\infty$ and let $\CC_\infty$ be the completion of an algebraic closure of $K_\infty$ at $\infty$. Put $G_K:=\Gal(K^\sep/K)$, the absolute Galois group of $K$. Let $$\varphi_t=\theta+a_1\tau+\cdots+a_r\tau^r$$ be a Drinfeld module over $K$, and let $\rho:G_K\longrightarrow \GL_n(\conj{\FF}_q)$ be an Artin representation. For $k\in \ZZ_{\geq 0}$, the twisted Goss $L$-series $L(\varphi^\vee,\rho,k)$ takes values in $K_\infty$ and has been studied in \cite{Ye26}. The transcendence of $L(\varphi^\vee,\rho,0)$ has been established in \cite[Theorem 5.2]{Ye26} and the purpose of the present article is to prove the transcendence of $L(\varphi^\vee,\rho,k)$ for all positive integers $k\geq 1$.

We are ready to state our main theorem.
\begin{thm}[Theorem \ref{thm:main}]\label{thm:int}
	 Let $\varphi$ be a Drinfeld module over $K$, and let $\rho:G_K\to \GL_n(\conj{\FF}_q)$ be an Artin representation. Then, for every non-negative integer $k$, $L(\varphi^\vee,\rho,k)$ is transcendental over $\conj K$.
\end{thm}
As a consequence, we have the following.
\begin{cor}[Corollary \ref{cor:car}]
	Let $\rho:G_K\to \GL_n(\conj\FF_q)$ be an Artin representation. Then, for
	every integer $k\geq1$, the special $L$-value $L(\rho,k)$ is transcendental
	over $\conj K$.
\end{cor}

Theorem \ref{thm:int} provides an evidence for the following conjecture.
\begin{cjt}\label{cjt:main}
	 Let $E$ be a finite extension of $K$ and $G_E = \Gal(E^\sep/E)$. Let $\varphi$ be a Drinfeld module over $E$, and let $\rho:G_E\to \GL_n(\conj{\FF}_q)$ be an Artin representation. Then, for every non-negative integer $k$, $L(\varphi^\vee,\rho,k)$ is transcendental over $\conj E$.
\end{cjt}

Our method uses the Artin twist $\mcE=\mcE(\varphi,\rho)$ constructed by the author in \cite{Ye26}. This is an abelian Anderson $t$-module over $K$ satisfying $$\MM(\mcE)\cong_K \MM(\varphi,\rho),$$ and, after choosing a fundamental solution of $\rho$ (see \cite{Ye26} for the definition), it becomes isomorphic over $K^\sep$ to a direct sum of copies of $\varphi$, i.e., we have an isomorphism of $t$-modules over $K^\sep$ $$\mcE\otimes_K K^\sep\cong \varphi^{\oplus nd}\otimes_K K^\sep,$$ where $n=\dim_{\conj{\FF}_q}\rho$ and $d = [\FF_q(\rho):\FF_q]$. For $k\geq 1$, we tensor the Artin-twisted motive $\MM(\varphi,\rho)$ with the $k$-th tensor power of the Carlitz motive and construct an Anderson $t$-module $\mcE_k=\mcE(\varphi,\rho,k)$ over $K$ such that $$\MM(\mcE_k)\cong_K\MM(\varphi,\rho)\otimes \MM(C)^{\otimes k}.$$
This construction shifts the $L$-series in the expected way: after removing finitely many local factors, we have $$L_S(\mcE_k^\vee,0)=N_{K_\infty(\rho)/K_\infty}\bigl(L_S(\varphi^\vee,\rho,k)\bigr)$$ for some finite set $S$ of places containing all the bad primes. Hence it suffices to prove the transcendence of $L(\mcE_k^\vee,0)$.

By Taelman's conjecture, which was proved in \cite{ANDTR20} for a large class of $t$-modules and is true for our $\mcE_k$ for all $k\geq 1$, $L(\mcE_k^\vee,0)$ is a rational multiple of the regulator of $\mcE_k$. The key point is to compare this regulator with another determinant of logarithms on a simpler $t$-module. With $G_k=G_{\varphi,k}$ (see \cite{GN25,GN24} or \S \ref{subsec:gn-prelim} for the definition), we have
$$\mcE_k\otimes_K K^\sep\cong G_k^{\oplus N}\otimes_K K^\sep,$$ where $N = nd$.

To carry out this comparison, we construct a transition matrix $\widetilde{\Pi}_k$ with algebraic entries such that $\mcE_k(t) = \widetilde{\Pi}_k^{-1}G_k(t)^{\oplus N}\widetilde{\Pi}_k$. The main difficulty is that the correct transition matrix is not obtained by repeating the transition matrix for $\mcE$ in every block. More precisely, by \cite{Ye26}, there exists a matrix $P\in \GL_N(K^\sep)$ with $$\mcE_t=P^{-1}\varphi_t^{\oplus N}P.$$ We need to write down $\widetilde{\Pi}_k$ in terms of $P$. In fact, put $P_\ell:=P^{(\ell-1)}$ for all $1\leq\ell\leq r$.  After regrouping the coordinates according to a certain index set, and setting $$\Pi_k=\operatorname{diag}\left(P_1,\ldots,P_r,\ldots,P_1,\ldots,P_r,P_1\right),$$ where the ordered block $(P_1,\ldots,P_r)$ is repeated $k$ times and every entry in the diagonal notation is an $N\times N$ block, the transition matrix to the ordinary direct sum can be obtained from $\Pi_k$ by a permutation of coordinates. Note that the Frobenius-twisted blocks are forced by the $\tau$-terms in the matrix defining $G_k$.

After this comparison, the regulator determinant becomes an algebraic multiple of a determinant whose entries are $\conj K$-linear forms in the tractable coordinates of logarithms on $G_k$. Then, applying Gezmi{\c s} and Namoijam's algebraic independence theorem \cite[Theorem 1.3]{GN25} for these tractable coordinates, we deduce that this determinant is transcendental. As a consequence, if we take $\rho$ to be the trivial representation, our result recovers Gezmi{\c s} and Namoijam's transcendence result \cite[Theorem 1.1(i)]{GN24} as a special case.

The paper is organized as follows. In Section \ref{sec: prelim}, we briefly review some preliminaries including Anderson $t$-modules, $t$-motives, $L$-series, Artin twists of Drinfeld modules, Taelman's regulators, and Gezmi{\c s} and Namoijam's $t$-module $G_k$. In Section \ref{sec:Ek-construction}, we construct the $t$-module $\mcE_k$ explicitly and discuss its properties. In Section \ref{sec:transcendence-Ek}, we prove Theorem \ref{thm:int} and discuss possible approaches to Conjecture \ref{cjt:main}, highlighting the main obstructions and explaining why our present method does not extend to the conjecture in full generality.

\textit{Acknowledgment.} The author would like to thank Yen-Tsung Chen, O{\u g}uz Gezmi{\c s}, Changningphabbi Namoijam, and Matthew Papanikolas for many helpful discussions, valuable suggestions, and generous assistance during the preparation of this work.
\section{Preliminaries}\label{sec: prelim}
\subsection{Notations}
Throughout the paper, we will use the following notation.
\begin{longtable}{l @{\hspace{1em}=\hspace{1em}} p{0.85\textwidth}}
	$p$ & a prime number in $\ZZ$.\\
	$\FF_q$ & a finite field with $q$ elements, where $q$ is a power of $p$.\\
	$\bA$ & $ \FF_q[t]$.\\
	$A$ & $ \FF_q[\theta]$.\\
	$K$ & $ \FF_q(\theta)$.\\
	$v_\infty$ & $-\deg$.\\
	$\inorm{\,\cdot\,}$ & the normalized $\infty$-adic norm on $K$ given by $\inorm{a} = q^{-d_\infty v_\infty(a)}$. \\
	$K_{\infty}$ & the completion of $K$ with respect to $\inorm{\,\cdot\,}$, i.e. $\laurent{\FF_q}{1/\theta}$. \\
	$\conj{K}_\infty$ & a fixed algebraic closure of $K_\infty$.\\
	$\CC_\infty$ & the completion of $\conj{K}_{\infty}$ with respect to $\inorm{\cdot}$, which is also algebraically closed. \\
	$\conj{K}$ & the algebraic closure of $K$ in $\CC_\infty$. \\
\end{longtable}

If $R$ is an $\FF_q$-algebra, we write $$R[\tau]=\left\{\sum_{i=0}^n a_i\tau^i\mid a_i\in R\right\}$$ for the twisted polynomial ring, in which the multiplication satisfies $\tau a=a^q\tau$ for all $a\in R$. For an element $\dis f = \sum_{i=1}^mf_it^i\in R[t]$, we define the Frobenius twists of $f$ by $$f^{(1)} = \sum_{i=1}^{m}f_i^{q} t^i.$$ For every positive integer $n > 0$, we define $f^{(n)} = (f^{(n-1)})^{(1)}$ recursively. We set $f^{(0)} = f$.  If, in addition, $R$ is perfect, i.e. $\tau: R\to R, a\mapsto a^q$ is an isomorphism, then this definition extends to all integers. For example, in this case, $\dis f^{(-1)} = \sum_{i=1}^{m}f_i^{1/q}t^i$. Let $\GG_{a,R}$ be the additive group scheme over $R$. It is well known that $R[\tau] \cong \End_{\FF_q}(\GG_{a,R})$, where the latter is the ring of all $\FF_q$-linear endomorphisms of $\GG_{a,R}$.

For a matrix $M=(m_{ij})$ with entries in $R[t]$, we write $$M^{(i)}=\left(m_{ij}^{(i)}\right)$$ for the $i$-th Frobenius twist, and if $R$ is perfect we also use $M^{(-i)}$ for the inverse Frobenius twist. 
\subsection{Anderson $t$-modules and $t$-motives}\label{subsec:anderson-prelim} In this section, we briefly review the theory of $t$-modules and $t$-motives introduced by Anderson \cite{anderson1986t}. We mainly follow \cite{hartl2020pinks} and \cite{Goss}.

Let $R$ be an $\bA$-algebra with the structure map $\iota:\bA\to R$. Under the standard identification $$\operatorname{Hom}_{\FF_q,R}(\GG_{a,R}^{\oplus n},\GG_{a,R}^{\oplus m}) \cong \Mat_{m\times n}(R)[\tau],$$ where $\tau$ is the Frobenius map raising an element to its $q$-th power, an $\FF_q$-linear homomorphism $\GG_{a,R}^{\oplus n}\to \GG_{a,R}^{\oplus m}$ of group schemes over $R$ can be identified with an element $\sum_i A_i\tau^i\in \Mat_{m\times n}(R)[\tau]$, whose tangent map at the neutral element is given by $\partial\left(\sum_i A_i\tau^i\right):=A_0$.

Let $\mathfrak{j}=(t-\iota(t))\subset R[t]$. For an $R[t]$-module $M$, set $\tau^*M:=R[t]\otimes_{\tau,R[t]}M$. An \emph{Anderson $t$-motive} over $R$ is a pair $(M,\tau_M)$, where $M$ is a projective $R[t]$-module and
$$\tau_M:\tau^*M[\mathfrak{j}^{-1}]\stackrel{\sim}{\longrightarrow}M[\mathfrak{j}^{-1}]$$ is an isomorphism of $R[t][\mathfrak{j}^{-1}]$-modules. It is called \emph{effective} if $\tau_M$ restricts to a morphism $\tau_M:\tau^*M\longrightarrow M$.  It is called \emph{abelian} if $M$ is finite projective over $R[t]$.

An \emph{Anderson $t$-module} over $R$ of dimension $d$ is an $\FF_q$-algebra homomorphism $\mcE:\bA\longrightarrow \End_{\FF_q,R}(\GG_{a,R}^{\oplus d})\cong \Mat_d(R)[\tau]$, such that $\left(\partial\mcE_t-\iota(t)I_d\right)^d=0$. An Anderson $t$-module of dimension $1$ is called a \emph{Drinfeld module}. Thus a Drinfeld module over $R$ has the form $$\varphi_t=\iota(t)+a_1\tau+\cdots+a_r\tau^r,$$ with $a_r\neq 0$ and $r$ is called its rank. 

Let $S$ be an $R$-algebra, then we have an $\bA$-module $\Lie_\mcE(S)$ whose underlying set is $S^d$ and the $\bA$-action is given by $$a \cdot \mathbf{x} = \partial \mcE_a(\mathbf{x}),$$ for all $a\in \bA$ and $\mathbf{x}\in \Lie_\mcE(S)$. The $\bA$-module scheme $\Lie_\mcE$ is called the \textit{Lie algebra} of $\mcE$. Similarly, $\mcE(S)$ of $S$-valued points has an $\bA$-module structure given by $$a \cdot \mathbf{x} =  \mcE_a(\mathbf{x}),$$ for all $a\in \bA$ and $\mathbf{x}\in \mcE(S)$.

If $R\subseteq \CC_\infty$ is a field, an Anderson $t$-module $\mcE$ admits an exponential map $\exp_\mcE:\Lie_\mcE(\CC_\infty)\longrightarrow \mcE(\CC_\infty)$, given by $$\exp_\mcE = \sum_{i=0}^{\infty}A_i\tau^i$$ with $A_i\in \Mat_{d}(\CC_\infty)$. It is $\FF_q$-linear, has identity as its constant term, and satisfies the functional equation $\exp_\mcE\circ \partial\mcE_a=\mcE_a\circ \exp_\mcE,$ for all $a\in \bA$. We say that $\mcE$ is \emph{uniformizable} if $\exp_\mcE$ is surjective. Its \emph{period lattice} is $\Lambda_\mcE:=\ker(\exp_\mcE)\subseteq \Lie_\mcE(\CC_\infty)$.

Let $\mcE$ be an Anderson $t$-module over $R$. Following Anderson, one attaches to $\mcE$ the $t$-motive $\MM(\mcE):=\operatorname{Hom}_{\FF_q,R}(\mcE,\GG_{a,R})$. It is an $R[t]$-module via $$(r\otimes a)\cdot m=r\cdot m\circ \mcE_a,\qquad r\in R,\ a\in \bA,$$ and it carries a $\tau$-semilinear action $\tau(m)=\Frob_{q,\GG_a}\circ m$. If $\MM(\mcE)$ is finite projective of rank $r$ over $R[t]$, then $\mcE$ is called \emph{abelian} and $r$ is called the rank of $\mcE$.

\begin{thm}[Anderson, Hartl {\cite[Theorem 3.5]{hartl2017isogenies}}]\label{thm:and}
The functor $$\mcE\longmapsto \MM(\mcE)$$ from the category of abelian Anderson $t$-modules over $R$ to that of effective $t$-motives over $R$ is contravariant and fully faithful. Its essential image consists of the effective $t$-motives which are finitely generated over $R[\tau]$ after a faithfully flat base change.
\end{thm}

Let $G_K:=\Gal(K^\sep/K)$. If $M$ is an abelian $t$-motive over $K$ and $\mfl$ is a finite prime of $A$, write $K_\mfl$ for the completion of $K$ at $\mfl$ and $H^1_\mfl(M,K_\mfl)$ for the $\mfl$-adic realization of $M$ with coefficients in $K_\mfl$. We denote the associated continuous Galois representation by
$$\rho_{M,\mfl}:G_K\longrightarrow \GL_{K_\mfl}\left(H^1_\mfl(M,K_\mfl)\right).$$

See \cite[\S 3.5]{hartl2020pinks} or \cite[\S 2.2.2]{Ye26} for details.
\subsection{Goss $L$-series and Carlitz shifts}
\label{subsec:L-series-shift}
We use the definition of Goss $L$-series from \cite[\S 3]{Ye26}. Let $M$ be an abelian $t$-motive over $K$. After removing a finite set $S$ of finite places, for each monic irreducible $\mfp\in A$ with $\mfp\notin S$ and for each prime $\mfl$ with $\mfl\neq \mfp$, let $$P_\mfp(M,X):=\det\left(X-\rho_{M,\mfl}(\Frob_\mfp)\mid H^1_\mfl(M,K_\mfl)^{I_\mfp}\right)$$ and let $$Q_\mfp(M,X):=\det\left(1-X\rho_{M,\mfl}(\Frob_\mfp)\mid H^1_\mfl(M,K_\mfl)^{I_\mfp}\right)$$ be the reciprocal polynomial of $P_\mfp(M,X)$. Thus, for every integer $s$, we defines $$L_S(M,s):=\prod_{\mfp\notin S}Q_\mfp(M,\mfp^{-s})^{-1},$$ where $\mfp$ also denotes its monic generator. If $G$ is an abelian Anderson $t$-module and $M=\MM(G)$, then the identity (3.1) in \cite[\S 3.1]{Ye26} gives \begin{equation}\label{eq:Ye26-module-motive-L}
L_S(M,s)=L_S(G^\vee,s).
\end{equation}
Thus all $L$-series of $t$-modules in this paper are understood through this convention.

\begin{thm}\label{thm:car-sh}
Let $G$ be an abelian Anderson $t$-module over $K$, put $M=\MM(G)$, and let $n\geq0$. Suppose that $G_n$ is an abelian Anderson $t$-module with
$$\MM(G_n)\cong M\otimes \MM(C)^{\otimes n}.$$

Then, for every finite set $S$ containing the bad places of both sides, $$L_S(G_n^\vee,0)=L_S(G^\vee,n).$$
\end{thm}

\begin{proof}
Put $\sfC:=\MM(C)$. After enlarging $S$, we may assume that $M$, $\sfC$,
and $M\otimes\sfC^{\otimes n}$ have good reduction at every $\mfp\notin S$.
For such $\mfp$, the inertia subgroup $I_\mfp$ acts trivially on the relevant $\mfl$-adic realizations; hence $H^1_\mfl(*,K_\mfl)^{I_\mfp} = H^1(*,K_\mfl)$, where $*\in \{M,\sfC,M\otimes \sfC^{\otimes n}\}$. Moreover, the compatibility of tensor product with $\mfl$-adic realization gives $$\rho_{M\otimes\sfC^{\otimes n},\mfl}(\Frob_\mfp)=\rho_{M,\mfl}(\Frob_\mfp)\otimes\rho_{\sfC,\mfl}(\Frob_\mfp)^{\otimes n}.$$
The Carlitz motive $\sfC=\MM(C)$ has the local factor at $\mfp$ given by $Q_\mfp(\sfC,X) = 1-\mfp^{-1}X$. Let $\alpha_1,\ldots,\alpha_d$ be the eigenvalues of $\rho_{M,\mfl}(\Frob_\mfp)$ on $H^1_\mfl(M,K_\mfl)$, counted with multiplicity. Then $$Q_\mfp(M,X)=\prod_{i=1}^d(1-\alpha_iX).$$
The eigenvalues of $\rho_{M\otimes\sfC^{\otimes n},\mfl}(\Frob_\mfp)$ are thus $\alpha_i\mfp^{-n}$ for $1\leq i\leq d$, and hence $$Q_\mfp(M\otimes\sfC^{\otimes n},X)=\prod_{i=1}^d(1-\alpha_i\mfp^{-n}X)=Q_\mfp(M,\mfp^{-n}X).$$
In particular, the local factor of $M\otimes \sfC^{\otimes n}$ at $s=0$ is  $$Q_\mfp(M\otimes\sfC^{\otimes n},1)^{-1}=Q_\mfp(M,\mfp^{-n})^{-1},$$ which is exactly the local factor of $M$ at $s=n$. Thus, by the definition of the $L$-series, $$L_S(M\otimes\sfC^{\otimes n},0)=\prod_{\mfp\notin S}Q_\mfp(M\otimes\sfC^{\otimes n},1)^{-1}=\prod_{\mfp\notin S}Q_\mfp(M,\mfp^{-n})^{-1}=L_S(M,n).$$
Since $\MM(G_n)\cong M\otimes\sfC^{\otimes n}$, this is $$L_S(\MM(G_n),0)=L_S(M,n).$$ Finally, applying \eqref{eq:Ye26-module-motive-L} to $G_n$ and to $G$ gives $$L_S(G_n^\vee,0)=L_S(G^\vee,n).$$
\end{proof}

\begin{rmk}\label{rmk:not}
The notation in \cite{GN24,GN25} is slightly different from the convention we use here. Their $L$-series are written with local factors $\det(1-X\Frob^{-1})$. Under our convention, the Carlitz motive is $\sfC=\MM(C)$ and its Frobenius is acting by $\mfp^{-1}$. Thus, tensoring by $\sfC$ in our notation corresponds to tensoring by the dual of the Carlitz motive in their convention. 
\end{rmk}

We now briefly recall the definitions of twisted $L$-series; one can also see \cite[\S 3]{Ye26}. Let $\rho:G_K\to \GL(V_\rho)$ be an Artin representation with coefficients in $\conj{\FF}_q$. For a finite prime $\mfp$ of $A$, let $I_\mfp$ be the inertia subgroup and let $\Frob_\mfp$ be the arithmetic Frobenius. Define $$Q_\mfp(\rho,X):= \det\left(1-X\rho(\Frob_\mfp)\mid V_\rho^{I_\mfp}\right).$$

For a finite set $S$ containing the ramified primes of $\rho$, we set
$$L_S(\rho,s):=\prod_{\mfp\notin S}Q_\mfp(\rho,\mfp^{-s})^{-1}.$$

Next let $\varphi$ be a Drinfeld module over $K$. Let $V_\mfl(\varphi)$ be its rational $\mfl$-adic Tate module and let $V_\mfl^\vee(\varphi)$ be the dual representation. After choosing an embedding $\FF_q(\rho)\hookrightarrow \overline K_\mfl$, put $$W^\vee_{\varphi,\rho,\mfl}:= \left(V_\mfl^\vee(\varphi)\otimes_{K_\mfl}\overline K_\mfl\right) \otimes_{\overline K_\mfl}\left(V_\rho\otimes_{\FF_q(\rho)}\overline K_\mfl\right).$$

Let $\Theta^\vee_{\varphi,\rho,\mfl}$ be the induced representation of $G_K$ on $W^\vee_{\varphi,\rho,\mfl}$. For every $\mfp\notin S$ and $\mfl\neq\mfp$, define $$Q_\mfp^\vee(\varphi,\rho,X):=\det\left(1-X\Theta^\vee_{\varphi,\rho,\mfl}(\Frob_\mfp)\mid (W^\vee_{\varphi,\rho,\mfl})^{I_\mfp}\right).$$
Then, $$L_S(\varphi^\vee,\rho,s):=\prod_{\mfp\notin S}Q_\mfp^\vee(\varphi,\rho,\mfp^{-s})^{-1}.$$

At primes where $\varphi$ has good reduction and $\rho$ is unramified, if $\alpha_i$ are the eigenvalues of $\rho_{\varphi,\mfl}^\vee(\Frob_\mfp)$ and $\beta_j$ are the eigenvalues of $\rho(\Frob_\mfp)$, then $$Q_\mfp^\vee(\varphi,\rho,X)=\prod_{i,j}(1-\alpha_i\beta_jX).$$
When $S = \varnothing$, we simply write $L$ instead of $L_\varnothing$.

\begin{lem}\label{lem:carL}
Let $S$ contain the ramified primes of $\rho$. Then, for every integer
$s$, we have $$L_S(C^\vee,\rho,s-1)=L_S(\rho,s).$$

Consequently, $L(C^\vee,\rho,s-1)$ and $L(\rho,s)$ differ by a finite product of non-zero algebraic local factors.
\end{lem}

\begin{proof}
Let $\mfp\notin S$. The Carlitz module has good reduction at $\mfp$, and the local factor at $\mfp$ is $$Q_\mfp^\vee(C,X)=1-\mfp^{-1}X.$$
Let $\beta_1,\ldots,\beta_n$ be the eigenvalues of $\rho(\Frob_\mfp)$ on $V_\rho$. By the definition of the twisted local factor, $$Q_\mfp^\vee(C,\rho,X)=\prod_{j=1}^n(1-\mfp^{-1}\beta_jX)=Q_\mfp(\rho,\mfp^{-1}X).$$

Evaluating at $X=\mfp^{-(s-1)}$ gives $$Q_\mfp^\vee(C,\rho,\mfp^{-(s-1)})
=Q_\mfp(\rho,\mfp^{-s}).$$

Thus the local factors of $L_S(C^\vee,\rho,s-1)$ and $L_S(\rho,s)$ are equal for all $\mfp\notin S$, and the results now follows.

\end{proof}

\subsection{Artin twists of Drinfeld modules}
\label{subsec:artin-twists-prelim}
We recall the construction of Artin twists of Drinfeld modules defined in \cite{Ye26}. Let $\rho:G_K\longrightarrow \GL_n(\conj{\FF}_q)$ be an Artin representation. Put $$\FF_q(\rho):=\FF_q(\rho(g)_{ij}:g\in G_K,\ 1\leq i,j\leq n)=\FF_{q^d}$$ and set $N:=nd$. Let $V_\rho$ be the corresponding $\FF_{q^d}$-representation of $G_K$. The $K$-vector space $\bD(V_\rho):=(K^{\sep}\otimes_{\FF_q}V_\rho)^{G_K}$ is an \'etale $\tau$-module over $K$ of dimension $N$. The Artin $t$-motive associated to $\rho$ is defined by $\MM(\rho):=\bD(V_\rho)\otimes_K K[t]$, with the diagonal $\tau$-action. If $\varphi$ is a Drinfeld module over $K$, its \emph{Artin-twisted motive} by $\rho$ is $\MM(\varphi,\rho):=\MM(\varphi)\otimes \MM(\rho)$.

For later use we recall a concrete basis for $\MM(\rho)$. Choose an $\FF_q$-basis $\vec{\boldsymbol{\alpha}}=(\alpha_1,\ldots,\alpha_d)^\top$ of $\FF_{q^d}$ and an $\FF_{q^d}$-basis $\bm{w} = (w_1,\cdots,w_n)$ of $V_\rho$. Define $\operatorname{Sol}_K(\rho,\vec{\boldsymbol{\alpha}},\bm{w})$ to be the $K$-vector space of matrices $\vec{\mathbf{x}}\in \Mat_{n\times d}(K^{\sep})$ such that $$g(\vec{\mathbf{x}})\vec{\boldsymbol{\alpha}}^{(\ell)}=\rho^{(\ell)}(g)^{-1}\vec{\mathbf{x}}\vec{\boldsymbol{\alpha}}^{(\ell)}$$ for all $g\in G_K$ and $0\leq \ell\leq d-1$, where $\rho^{(\ell)}$ denotes the $\ell$-th Frobenius twist of $\rho$. We often write $\operatorname{Sol}_K(\rho)$ when the bases $\vec{\boldsymbol{\alpha}}$ and $\bm{w}$ are fixed. This is an \'etale $\tau$-module over $K$ of dimension $N$, and $$\MM(\rho)\cong \operatorname{Sol}_K(\rho)\otimes_K K[t].$$
An element $\vec{\mathbf{u}}\in \operatorname{Sol}_K(\rho)$ is called a \emph{fundamental solution} if all its entries are linearly independent over $\FF_q$. Such elements exist, and if $\vec{\mathbf{u}}$ is fundamental, then $\vec{\mathbf{u}},\vec{\mathbf{u}}^{(1)},\ldots,\vec{\mathbf{u}}^{(N-1)}$ is a $K$-basis of $\operatorname{Sol}_K(\rho)$.

We treat each solution $\vec{\mathbf{x}}\in \operatorname{Sol}(\rho)$ as a column vector in $(K^\sep)^N$. Let $M(\vec{\mathbf{u}}):=[\vec{\mathbf{u}},\vec{\mathbf{u}}^{(1)},\ldots,\vec{\mathbf{u}}^{(N-1)}]$. There are unique $f_0,\ldots,f_{N-1}\in K$ such that
$$\vec{\mathbf{u}}^{(N)}=f_0\vec{\mathbf{u}}+f_1\vec{\mathbf{u}}^{(1)}+\cdots+f_{N-1}\vec{\mathbf{u}}^{(N-1)}.$$
With respect to the fundamental basis above, the $\tau$-action on $\MM(\rho)$ is represented by $$\Phi_\rho(\vec{\mathbf{u}})=\begin{pmatrix}
0 & \cdots & 0 & f_0\\
1 & \cdots & 0 & f_1\\
\vdots & \ddots & \vdots & \vdots\\
0 & \cdots & 1 & f_{N-1}
\end{pmatrix}
=M(\vec{\mathbf{u}})^{-1}M(\vec{\mathbf{u}})^{(1)}.
$$

\begin{thm}[Ye {\cite[\S 4]{Ye26}}]\label{thm:AT}
Let $\varphi:\bA\to K[\tau]$ be a Drinfeld module of rank $r$ given by 
\begin{equation}
	\varphi_t=\theta+a_1\tau+\cdots+a_r\tau^r,\qquad a_r\neq 0,
\end{equation}
and let $\rho:G_K\to \GL_n(\conj{\FF}_q)$ be an Artin representation. Then we have
\begin{enumerate}
\item There exists an abelian Anderson $t$-module $\mcE=\mcE(\varphi,\rho)$ over $K$ of dimension $N=nd$ and rank $rN$ such that $\MM(\mcE)\cong \MM(\varphi,\rho):=\MM(\varphi)\otimes \MM(\rho)$.
\item Fix a fundamental solution $\vec{\mathbf{u}}\in \operatorname{Sol}_K(\rho)$ and set
$\Psi_\rho:=(\Phi_\rho(\vec{\mathbf{u}})^\top)^{-1}$, the Anderson $t$-module $\mcE$ has a model over $K$ given by $$\mcE_t=\theta I_N+a_1\Psi_\rho\tau+a_2\Psi_\rho\Psi_\rho^{(1)}\tau^2+\cdots
+a_r\Psi_\rho\Psi_\rho^{(1)}\cdots\Psi_\rho^{(r-1)}\tau^r.$$
\item Set $P = P(\vec{\mathbf{u}}) = (M(\vec{\mathbf{u}})^\top)^{-1}$. Then, $\mcE_t = P^{-1}\varphi^{\oplus N}_tP$. In particular, $\mcE$ is isomorphic to $\varphi^{\oplus N}$ over $K^{\sep}$ and $\mcE$ is hence uniformizable.
\item There exists a finite set $S$ of places such that $$L_S(\mcE^\vee, k) = N_{K_\infty(\rho)/K_\infty}(L_S(\varphi^\vee,\rho,k))$$ for all non-negative integers $k$.
\end{enumerate}

\end{thm}

We call $\mcE = \mcE(\varphi,\rho)$ the \emph{Artin twist} of $\varphi$ by $\rho$.
\subsection{Regulators and Taelman class number formula}
\label{subsec:taelman-gn-input}

Our exposition follows from \cite{ANDTR20,angles2022class} and \cite{taelman2012special}. We use the fixed isomorphism $\iota:\bA\to A$, $t\mapsto\theta$, whenever a $\bA$-module is viewed as an $A$-module.

Let $R$ be a commutative ring and let $M$ be a finitely presented $R$-module. Choose a finite presentation $$R^a\xrightarrow{T}R^b\longrightarrow M\longrightarrow0.$$
The Fitting ideal $\Fitt_R(M)$ is the ideal generated by the $b\times b$ minors of the matrix $T$ if $b\leq a$; if $b>a$, we set $\Fitt_R(M)=0$. This ideal is independent of the chosen presentation. If $M$ is a finite $A$-module, then $\Fitt_A(M)$ is principal, and we denote by $[M]_A$ its monic generator. For a finite $\bA$-module $M$, we use the same notation after applying $\iota$ to the monic generator of $\Fitt_{\bA}(M)$.

Let $V$ be a finite-dimensional $K_\infty$-vector space. An $A$-submodule $M\subset V$ is an $A$-lattice if it is discrete and spans $V$ over $K_\infty$. If $M,N\subset V$ are $A$-lattices, choose $A$-bases $e_1,\ldots,e_n$ of $M$ and $f_1,\ldots,f_n$ of $N$, where $n=\dim_{K_\infty}V$. Let $\gamma:V\to V$ be the $K_\infty$-linear map satisfying $\gamma(e_i)=f_i$. The ratio of covolumes $$[M:N]_A:=\frac{\det(\gamma)}{\sgn(\det(\gamma))}.$$
It does not depend on the choice of bases. The ratio of covolumes satisfies
$$[M:N]_A=[N:M]_A^{-1},\qquad[M_1:M_3]_A=[M_1:M_2]_A[M_2:M_3]_A,$$
and, if $N\subset M$, then $[M:N]_A = [M/N]_A$.

Let $G=(\GG_{a/K}^d,\phi)$ be an abelian $t$-module with $\phi_t\in \Mat_d(A)[\tau]$. We write $$U(G/A):=\{x\in \Lie(G)(K_\infty)\mid \exp_G(x)\in G(A)\}$$ for its \textit{unit module}, and $$H(G/A):=G(K_\infty)/(G(A)+\exp_G(\Lie(G)(K_\infty)))$$
for its \textit{class module}. We further define $$W_G(K_\infty):=\Lie(G)(K_\infty)/(\partial\phi_t-\theta I_d)\Lie(G)(K_\infty).$$
Let $W_G(A)$ be the image of $\Lie(G)(A)$ in $W_G(K_\infty)$.

\subsection{Gezmi{\c s} and Namoijam's $t$-modules $G_k$}
\label{subsec:gn-prelim}
We briefly review the construction of $t$-module $G_{\varphi,k}$ constructed by Gezmi{\c s} and Namoijam in \cite{GN24,GN25} and an algebraic independence result that will be used in the proof of Theorem \ref{thm:int}.

\begin{Def}[Gezmi{\c s}-Namoijam {\cite{GN24,GN25}}]\label{def:Gk}
Let $$\varphi_t=\theta+a_1\tau+\cdots+a_r\tau^r$$ be a Drinfeld module defined over $K$. For each $k\geq1$, we define the $t$-module $G_{\varphi,k}$ by $$G_{\varphi,k}(t)=\theta I_{rk+1}+N_k+B_k\tau,$$ where $$N_k =
\begin{pmatrix}
	0 & \cdots & 0 & \braceover{$rk+1-r$}{1&0& \cdots & 0}   \\
	& \ddots &   & \ddots & \ddots & & \vdots \\
	&        & \ddots & & \ddots & \ddots & 0 \\
	&        &        & 0 & \cdots & 0 & 1 \\
	&        &        &   & 0 & \cdots & 0 \\
	&        &        &   &   & \ddots & \vdots \\
	&        &        &   &   &        & 0
\end{pmatrix}\!
\begin{aligned}
	&\left. \vphantom{\begin{matrix} 0 \\ \ddots \\ \ddots \\ 0 \end{matrix}} \right\} rk+1-r \\
	&\left. \vphantom{\begin{matrix} 0 \\ \ddots \\ 0 \end{matrix}} \right\} r
\end{aligned}
$$ and $$B_k=\begin{pmatrix}
	0&\cdots&\cdots&\cdots&\cdots & \cdots & 0\\
	\vdots& & & & & & \vdots\\
	0&\cdots&\cdots&\cdots&\cdots &\cdots &0\\
	1&0&\cdots&\cdots&\cdots & \cdots & 0\\
	0&\ddots&\ddots&&&&\vdots\\
	\vdots&&1&\ddots & & &\vdots\\
	a_1&\cdots &\cdots&a_r&0&\cdots&0
\end{pmatrix}
\begin{aligned}
	&\left. \vphantom{\begin{matrix} 0 \\ \vdots \\ 0 \end{matrix}} \right\} rk+1-r \\
	&\left. \vphantom{\begin{matrix} 1 \\ 0 \\ \vdots \\ a_1 \end{matrix}} \right\} r.
\end{aligned}$$
When $\varphi$ is clear from the context, we write $G_k$ instead of $G_{\varphi,k}$.
\end{Def}

\begin{eg}\label{eg:Gk-r1}
Following \cite[Example 2.2]{GN24}, let $\mfb\in K^\times$, and let $n\geq1$. The \textit{generalized Carlitz $n$-th tensor power} $C_n^{(\mfb)}=(\GG_{a,L}^n,C_n^{(\mfb)})$ is the $t$-module determined as follows. If $n=1$, then $$C_1^{\mfb}(t) := C^{(\mfb)}(t)=\theta+\mathfrak b\tau.$$

If $n\geq2$, then $$C_n^{(\mfb)}(t):=\begin{pmatrix}
\theta&1&0&\cdots&0\\
0&\theta&1&\ddots&\vdots\\
\vdots&\ddots&\ddots&\ddots&0\\
0&\cdots&0&\theta&1\\
0&\cdots&\cdots&0&\theta
\end{pmatrix}+\begin{pmatrix}
0&0&\cdots&0&0\\
0&0&\cdots&0&0\\
\vdots&\vdots&&\vdots&\vdots\\
0&0&\cdots&0&0\\
\mathfrak b&0&\cdots&0&0
\end{pmatrix}\tau.$$
In particular, $C_n^{(1)}=C^{\otimes n}$, where $C$ is the Carlitz module.
Their $t$-motives satisfies
\begin{equation}\label{eq:Cnb-motive}
\MM(C_{n+1}^{(\mfb)})\cong\MM(C^{(\mfb)})\otimes\MM(C)^{\otimes n};
\end{equation}
see \cite[Example 2.2]{GN24}. In particular,
$C_n^{(\mfb)}$ has dimension $n$ and rank one.

If $r=1$, then $\varphi=C^{(a_1)}$ and Definition \ref{def:Gk} gives $G_{\varphi,k}=C_{k+1}^{(a_1)}$.
\end{eg}

\begin{prop}[Gezmi{\c s}--Namoijam]\label{prop:Gk-mot}
For every $r\geq1$, $$\MM(G_{\varphi,k})\cong\MM(\varphi)\otimes\MM(C)^{\otimes k}.$$
Moreover, $G_{\varphi,k}$ has dimension $rk+1$ and rank $r$.
\end{prop}

\begin{proof}
See \cite[\S 2.4]{GN24}. 
\end{proof}

We consider the projection map $$p_k:\Lie_{G_k}(\CC_\infty)\longrightarrow\CC_\infty^r$$ given by $$\bz = (z_1,\ldots,z_{rk+1})^\top\longmapsto(z_{r(k-1)+2},\ldots,z_{rk+1})^\top.$$

\begin{rmk}\label{rmk:trac}
	The entries of $p_k(\bz)$ are \textit{tractable coordinates} of $G_k$ in the sense of \cite[Definition 3.3.1]{CM21}, i.e. the $i$-th
	coordinate is tractable if, for every $a\in\FF_q[t]$, the $i$-th
	coordinate of the Lie action of $a$ on $\bz$ is $a(\theta)z_i$. Indeed, the Lie action of $t$ is given by $$T_k:=\partial G_k(t)=\theta I_{rk+1}+N_k.$$
	By the definition of $N_k$, its last $r$ rows are zero. Hence, for
	each $i=r(k-1)+2,\ldots,rk+1$, the $i$-th coordinate of $T_k\bz$ is
	$\theta z_i$. 
\end{rmk}

Since $\partial G_k(t)-\theta I_{rk+1}=N_k$, one has $$N_k\Lie_{G_k}(\CC_\infty)=\ker(p_k).$$
Thus $p_k$ identifies $$W_{G_k}(\CC_\infty)=\Lie_{G_k}(\CC_\infty)/(\partial G_k(t)-\theta I_{rk+1})\Lie_{G_k}(\CC_\infty)$$ with $\CC_\infty^r$.

As an analogue of \cite[Theorem 1.1]{CP12}, Gezmi{\c s} and Namoijam proved the following result.
\begin{thm}[Gezmi{\c s}--Namoijam {\cite[Theorem 1.3]{GN25}}]\label{thm:Gk-ai}
Let $\varphi$ be a Drinfeld module of rank $r$ over $K$, let $k\geq1$, and let $G_k=G_{\varphi,k}$. Put $F_k:=\Frac(\End_{\conj K}(G_k))$. Suppose that  $\bz_1,\ldots,\bz_m\in\Lie_{G_k}(\CC_\infty)$ satisfy $\exp_{G_k}(\bz_i)\in G_k(\conj K)$ for all $1\leq i\leq m$. If $\bz_1,\ldots,\bz_m$ are linearly independent over $F_k$, then the tractable coordinates of $\bz_1,\cdots,\bz_n$, i.e. all entries of $p_k(\bz_1),\ldots,p_k(\bz_m)$, are algebraically independent over $\conj K$.
\end{thm}

\section{The $t$-module $\mcE_k=\mcE(\varphi,\rho,k)$}\label{sec:Ek-construction}
In this section, we fix a Drinfeld module $\varphi$ and an Artin representation $\rho$ as before. For each $k\geq1$, define the $t$-motive over $K$ by
$$\MM(\varphi,\rho,k):=\MM(\varphi,\rho)\otimes \MM(C)^{\otimes k}.$$
Let $\mcE=\mcE(\varphi,\rho)$ be the Anderson $t$-module of dimension $N$ over $K$ in Theorem \ref{thm:AT} with $\MM(\mcE)\cong \MM(\varphi,\rho)$. Inspired by the construction of $G_n$ in \cite{GN24, GN25}, we will construct an Anderson $t$-module $\mcE_k = \mcE(\varphi,\rho,k)$ over $K$ satisfying
$$\MM(\mcE_k)\cong \MM(\varphi,\rho,k).$$ We use $k$ here instead of $n$ as $n$ was reserved for the dimension of $\rho$.

Write $$\mcE_t=\theta I_N+A_1\tau+\cdots+A_r\tau^r,$$ where $A_i=a_i\Psi_\rho\Psi_\rho^{(1)}\cdots \Psi_\rho^{(i-1)}\in \Mat_N(K)$ for each $1\leq i\leq r$.

Let $m_1,\ldots,m_N$ be the basis of the left $K[\tau]$-module $\MM(\mcE)=\operatorname{Hom}_{\FF_q,K}(\mcE,\GG_a)$ consisting of coordinate projections, i.e. $$m_i: \begin{pmatrix}
x_1\\
\vdots\\
x_N
\end{pmatrix}\mapsto x_i.$$ Then, by the definition of the $K[t]$-module structure on $\MM(\mcE)$, for each $1\leq i\leq N$, we have
\begin{equation}\label{t-action}
	(t-\theta)m_i=\sum_{\ell=1}^r\sum_{j=1}^N (A_\ell)_{ij}\tau^\ell m_j.
\end{equation}

For each $1\leq \ell\leq r,\ 1\leq i\leq N$, we put $$m_{\ell,i}:=\tau^{\ell-1}m_i,$$ so that $\{m_{\ell,i}\}$ is a $K[t]$-basis of $\MM(\mcE)$. Let $\widetilde{m}$ be the standard basis of $\MM(C)^{\otimes k}$ satisfying
\begin{equation}\label{tau-action-on-carlitz}
\tau\widetilde{m}=(t-\theta)^k\widetilde{m}.
\end{equation}
In the tensor product $\MM(\mcE)\otimes \MM(C)^{\otimes k}$, for each $1\leq \ell\leq r,\ 1\leq i\leq N$ and $0\leq j\leq k-1$, we set
$$v_{\ell,i,j}:=m_{\ell,i}\otimes (t-\theta)^j\widetilde{m},$$
and also $$v_{1,i, k}:=m_{1,i}\otimes (t-\theta)^k\widetilde{m}.$$
Note that the $t$-motive $\MM(\mcE)\otimes\MM(C)^{\otimes k}$ has dimension $d_k:=N(rk+1)$ and rank $r_k = rN$ by \cite[Definition 5.5.8, Proposition 5.7.2]{Goss}. It is easy to check these elements form a $K[\tau]$-basis of $\MM(\mcE)\otimes \MM(C)^{\otimes k}$ and for all $1\leq i\leq N$, they satisfy 
\begin{equation}\label{t-theta-actions}
	\begin{cases}
		(t-\theta)v_{\ell,i, j}=v_{\ell, i,j+1}, \quad & 1\leq \ell\leq r,\ 0\leq j\leq k-2,\ \\
		(t-\theta)v_{1,i, k-1}=v_{1,i, k},\quad & \ell = 1,\ j = k-1 \\
		(t-\theta)v_{\ell,i, k-1}=\tau v_{\ell-1,i,0},\quad & 2\leq \ell\leq r,\ j = k-1\\
		(t-\theta)v_{1,i, k}=\sum\limits_{\ell=1}^r\sum\limits_{h=1}^N (A_\ell)_{i h}\tau v_{\ell,h, 0}, \quad & \ell =1, \ j = k.
	\end{cases}
\end{equation}

Indeed, the first two identities are immediate from the $K[t]$-module structure, the third follows from (\ref{tau-action-on-carlitz}), and the last one follows from (\ref{t-action}).

We write matrices in $\Mat_{d_k}(K)$ as $(rk+1)\times (rk+1)$ block matrices with blocks of size $N\times N$. Define two such block matrices $\mcU_k$ and $\mcA_k$ by 
\ \\

$$\mcU_k = \begin{pmatrix}
	0 & \cdots & 0 & \braceover{$rk+1-r$}{I_N&0& \cdots & 0}   \\
	& \ddots &   & \ddots & \ddots & & \vdots \\
	&        & \ddots & & \ddots & \ddots & 0 \\
	&        &        & 0 & \cdots & 0 & I_N \\
	&        &        &   & 0 & \cdots & 0 \\
	&        &        &   &   & \ddots & \vdots \\
	&        &        &   &   &        & 0
\end{pmatrix}\!
\begin{aligned}
	&\left. \vphantom{\begin{matrix} 0 \\ \ddots \\ \ddots \\ 0 \end{matrix}} \right\} rk+1-r \\
	&\left. \vphantom{\begin{matrix} 0 \\ \ddots \\ 0 \end{matrix}} \right\} r
\end{aligned}$$ and $$\mcA_k=
\begin{pmatrix}
0&\cdots&\cdots&\cdots&\cdots & \cdots & 0\\
\vdots& & & & & & \vdots\\
0&\cdots&\cdots&\cdots&\cdots &\cdots &0\\
I_N&0&\cdots&\cdots&\cdots & \cdots & 0\\
0&\ddots&\ddots&&&&\vdots\\
\vdots&&I_N&\ddots & & &\vdots\\
A_1&\cdots &\cdots&A_r&0&\cdots&0
\end{pmatrix}
\begin{aligned}
	&\left. \vphantom{\begin{matrix} 0 \\ \vdots \\ 0 \end{matrix}} \right\} rk+1-r \\
	&\left. \vphantom{\begin{matrix} I_N \\ 0 \\ \vdots \\ A_1 \end{matrix}} \right\} r
\end{aligned}
$$
where the braces indicate the first $rk+1-r$ block rows and the last $r$ block rows. We define
\begin{equation}\label{twist_k}
	\mcE_k(t):=\theta I_{d_k}+\mcU_k+\mcA_k\tau.
\end{equation}

Equivalently, if we order the $K[\tau]$-basis as
$$ \bm{v} :=\bigl(v_{1,0},\ldots,v_{r,0},\ldots,v_{1,k-1},\ldots,v_{r,k-1},v_{1,k}\bigr)^{\top},$$ where $v_{\ell,j}:=(v_{\ell,1,j},\ldots,v_{\ell,N,j})^{\top}$, then (\ref{t-theta-actions}) implies that $$t\cdot\bm{v}=\mcE_k(t) \cdot \bm{v}.$$
Thus we see that $\mcE(\varphi,\rho,k):=\mcE_k$ defined in (\ref{twist_k}) is the $t$-module corresponding to the $t$-motive $\MM(\varphi,\rho,k)$, i.e. $\MM(\mcE_k) = \MM(\varphi,\rho,k)$.

\begin{rmk}\label{rmk:ana}
	It is the analogue of the $t$-module $G_n$ in \cite{GN24, GN25}, with each scalar entry $1$ replaced by $I_N$ and the last-row coefficients $a_i$ replaced by the matrix coefficients $A_i$ of the Artin twist $\mcE$.
\end{rmk}

\begin{cor}\label{cor:Ek-L}
For every $k\geq1$ and every finite set $S$ containing the bad places,
\begin{equation}\label{eq:Ek-L-shift}
L_S(\mcE_k^\vee,0)=L_S(\mcE^\vee,k)
=N_{K_\infty(\rho)/K_\infty}\bigl(L_S(\varphi^\vee,\rho,k)\bigr).
\end{equation}
\end{cor}
\begin{proof}
By construction, $\MM(\mcE_k)\cong\MM(\mcE)\otimes\MM(C)^{\otimes k}$. Theorem \ref{thm:car-sh} gives
$$
L_S(\mcE_k^\vee,0)=L_S(\mcE^\vee,k).
$$
The second equality is Theorem \ref{thm:AT}(4), evaluated at $s=k$.
\end{proof}

\begin{eg}\label{eg:Ek-r1}
Assume that $\varphi$ has rank one, so that $$\varphi_t=\theta+a_1\tau.$$
Then $d_k=N(k+1)$ and the construction in \eqref{twist_k} gives
$$\mcE_k(t)=\begin{pmatrix}
\theta I_N&I_N&0&\cdots&0\\
0&\theta I_N&I_N&\ddots&\vdots\\
\vdots&\ddots&\ddots&\ddots&0\\
0&\cdots&0&\theta I_N&I_N\\
0&\cdots&\cdots&0&\theta I_N
\end{pmatrix}+\begin{pmatrix}
0&0&\cdots&0&0\\
0&0&\cdots&0&0\\
\vdots&\vdots&&\vdots&\vdots\\
0&0&\cdots&0&0\\
A_1&0&\cdots&0&0
\end{pmatrix}\tau,$$
where $A_1=a_1\Psi_\rho$. 
In particular, when $N=1$, this is the $t$-module
$C_{k+1}^{(A_1)}$ from \cite[Example 2.2]{GN24}.
\end{eg}

\section{Transcendence of $L(\mcE_k^\vee,0)$}\label{sec:transcendence-Ek}

In this section, we will prove Theorem \ref{thm:int}, i.e. for a Drinfeld module $\varphi$ defined over $K$ and an Artin representation $\rho:G_K\to \GL_n(\conj \FF_q)$, the special $L$-value $L(\varphi^\vee,\rho,k)$ is transcendental over $K$. The case $k=0$ is true by \cite[Theorem 5.2]{Ye26}, so we only need to deal with positive integers $k$. The main idea is to construct a matrix $\widetilde{\Pi}_k$ with algebraic entries such that $\mcE_k(t) = \widetilde{\Pi}_k^{-1}G_k(t)^{\oplus N}\widetilde{\Pi}_k$. Then, studying the algebraic relations of the logarithms of $\mcE_k$ is equivalent to studying the algebraic relations of logarithms of $G_k^{\oplus N}$, which enbles us to apply the result of Gezmi{\c s} and Namoijam.
\subsection{Some lemmas}
We keep the notation of Subsection \ref{subsec:gn-prelim}; in particular $G_k=G_{\varphi,k}$ and $p_k$ is the projection map that takes $\bz$ to its tractable coordinates. Let $R_{k}:=\End_{\conj K}(G_{k})$.
\begin{lem}\label{lem:end}
	For each $\eta\in\End_{\conj K}(\varphi)$, let $\eta_k$ be the unique endomorphism of $G_{k}$ whose induced endomorphism on the $t$-motive $\MM(G_{k})\cong\MM(\varphi)\otimes_{\conj K[t]}\MM(C)^{\otimes k}$ is $\eta_k^*=\eta^*\otimes\id_{\MM(C)^{\otimes k}}$. Then the map
	$$\End_{\conj K}(\varphi)\longrightarrow R_k,$$ given by $$\eta\longmapsto \eta_k$$ is an isomorphism of rings.
\end{lem}
\begin{proof}
	This is essentially a restatement of \cite[Proposition 4.3]{GN25}. We present the proof here for completeness. Choose a $\conj K[t]$-basis $\bm{m}=(m_1,\ldots,m_r)^\top$ of $\MM(\varphi)$ and write $\tau\bm{m}=\Phi\bm{m}$ with $\Phi\in\Mat_r(\conj K[t])$. Let $\widetilde m_k$ be a $\conj K[t]$-basis of $\MM(C)^{\otimes k}$ satisfying $\tau\widetilde m_k=(t-\theta)^k\widetilde m_k$. Thus, $$\bm{m}_k:=(m_1\otimes\widetilde m_k,\ldots,m_r\otimes\widetilde m_k)^\top$$ is a $\conj K[t]$-basis for $\MM(G_{k})$, and $\tau\bm{m}_k=(t-\theta)^k\Phi\bm{m}_k$.
	
	Let $f$ be a $\conj K[t]$-linear endomorphism of $\MM(\varphi)$, and let $F\in\Mat_r(\conj K[t])$ be determined by $f(\bm{m})=F\bm{m}$. Thus, if $$f(m_i)=\sum_{j=1}^r F_{ij}m_j,$$ then $$(f\otimes\id_{\mcM(C)^{\otimes k}})(m_i\otimes\widetilde m_k)=f(m_i)\otimes\widetilde m_k=\sum_{j=1}^rF_{ij}(m_j\otimes\widetilde m_k).$$
	Thus, $f\otimes\id_{\MM(C)^{\otimes k}}$ is represented by the same matrix $F$ with respect to the basis $\bm{m}_k$.
	
	Since $\tau$ is Frobenius semilinear, the equality $f\tau=\tau f$ is
	equivalent to
	\begin{equation}\label{eq:endomorphism-matrix-psi}
		F^{(1)}\Phi=\Phi F.
	\end{equation}
	On the other hand, let $g$ be a $\conj K[t]$-linear endomorphism of
	$\MM(G_{k})$, and let $Q\in\Mat_r(\conj K[t])$ be determined by
	$g(\bm{m}_k)=Q\bm{m}_k$. Then $g\tau = \tau g$ if and only if
	\begin{equation}\label{eq:endomorphism-matrix-Gk}
		Q^{(1)}(t-\theta)^k\Phi
		=(t-\theta)^k\Phi Q.
	\end{equation}
	Since $(t-\theta)^k$ is a nonzero scalar in the integral domain $\conj K[t]$, equation \eqref{eq:endomorphism-matrix-Gk} is equivalent to $Q^{(1)}\Phi=\Phi Q$. 
	Define $$f_Q(\bm{m}):=Q\bm{m},$$ then we have
	$f_Q\in\End_{\conj K[t,\tau]}(\MM(\varphi))$. The endomorphisms $g$ and
	$f_Q\otimes\id_{\MM(C)^{\otimes k}}$ have the same representing matrix
	$Q$ in the basis $\bm{m}_k$, so we see that $g=f_Q\otimes\id_{\MM(C)^{\otimes k}}$. 
	This shows that $$\End_{\conj K[t,\tau]}(\MM(\varphi))\stackrel{\sim}{\rightarrow}\End_{\conj K[t,\tau]}(\MM(G_{k})),$$ $$f\mapsto f\otimes\id_{\MM(C)^{\otimes k}},$$ is an isomorphism. 
	
	By the full faithfulness of the functor $\mcE\mapsto \MM(\mcE)$ in Theorem
	\ref{thm:and}, this isomorphism gives an
	isomorphism $\End_{\conj K}(\varphi)\longrightarrow
	\End_{\conj K}(G_{k})$, $\eta\longmapsto\eta_k$.
\end{proof}

In particular, $R_{k}$ is a commutative integral domain and we may put $F_{k}:=\Frac(R_{k})$. We now specify the action of $F_{k}$ on $\Lie(G_{k})(\CC_\infty)$.

We first observe that $\partial v$ is invertible for every nonzero $v\in R_{k}$. Indeed, write $v=\eta_k$ with $0\neq\eta\in\End_{\conj K}(\varphi)$. Let $\mu$ be the dual isogeny of $\eta$ such that  $\mu\eta=\varphi_a$ for some $0\neq a\in\bA$. Then we have $\mu_kv=\mu_k\eta_k = G_{k}(a)$.  Set $$N_{k}:=\partial G_{k}(t)-\theta I_{rk+1}.$$ Taking differentials gives $$(\partial\mu_k)(\partial v)=\partial G_{k}(a)=a(\partial G_{k}(t))=a(\theta I_{rk+1}+N_{k}).$$ Since $N_{k}$ is nilpotent, we have $\det\bigl(a(\theta I_{rk+1}+N_{k})\bigr)=a(\theta)^{rk+1}\neq 0$. It follows that $\partial v\in\GL_{rk+1}(\conj K)$.

For $u,v\in R_{k}$ with $v\neq0$, define
\begin{equation}\label{eq:F-action-Lie}
	\frac{u}{v}\cdot\bz:=(\partial v)^{-1}(\partial u)\bz,
\end{equation} for every $\bz\in\Lie_{G_{k}}(\CC_\infty)$.
This action is well-defined and does not depend on the choice of $u$ and $v$.

For every $h\in R_{k}$, we have $hG_{k}(t)=G_{k}(t)h$. Thus, $(\partial h)N_{k}=N_{k}(\partial h)$. Consequently, for every $c=u/v\in F_{k}$, the matrix $\partial_k(c):=(\partial v)^{-1}(\partial u)\in\Mat_{rk+1}(\conj K)$ commutes with $N_{k}$ and therefore preserves $N_{k}\Lie_{G_{k}}(\CC_\infty)=\ker(p_k)$. Thus, $$N_{k}\CC_\infty^{rk+1}=\CC_\infty^{r(k-1)+1}\oplus\{0\}.$$ With respect to the decomposition $$ \CC_\infty^{rk+1}=\CC_\infty^{r(k-1)+1}\oplus\CC_\infty^r,$$ there is consequently a unique matrix $D_k(c)\in\Mat_r(\conj K)$ such that
\begin{equation}\label{eq:F-action-block}
	\partial_k(c)= \begin{pmatrix}
		* & *\\
		0 & D_k(c)
	\end{pmatrix}.
\end{equation}
Equivalently, $D_k(c)$ is characterized by
\begin{equation}\label{eq:F-action-tractable}
	p_k(c\cdot\bz)=D_k(c)p_k(\bz),
\end{equation} for every $\bz\in\Lie_{G_{k}}(\CC_\infty)$.

\begin{lem}\label{lem:Gk-lin}
	Let $k\geq 1$ and keep the notation of Theorem \ref{thm:Gk-ai}. Let $\bz_1,\cdots,\bz_M\in \Lie_{G_{k}}(\CC_\infty)$ satisfy $\exp_{G_{k}}(\bz_i)\in G_{k}(\conj K)$. Choose a maximal subset $\bz_{i_1},\cdots,\bz_{i_m}$ which is linearly independent over $F_{k}$. Then every coordinate of $p_k(\bz_i)$ is a $\conj K$-linear form in the coordinates of $p_k(\bz_{i_1}),\cdots,p_k(\bz_{i_m})$.
\end{lem}

\begin{proof}
	By maximality, for each $1\leq i\leq M$ and $1\leq \mu\leq m$, there exist elements $a_{\mu i}\in F_{k}$ such that $\dis \bz_i=\sum_{\mu=1}^m a_{\mu i}\cdot \bz_{i_\mu}$ in $\Lie_{G_k}(\CC_\infty)$. Applying \eqref{eq:F-action-tractable} to the above $F_{k}$-linear expression gives
	$$p_k(\bz_i)=\sum_{\mu=1}^mp_k(a_{\mu i}\cdot \bz_{i_\mu}) = \sum_{\mu=1}^m D_k(a_{\mu i})p_k(\bz_{i_\mu}).$$ Equivalently, if $p_k(\bz_{i_\mu})=(z_{\mu,1},\ldots,z_{\mu,r})^\top$,  then every coordinate of $p_k(\bz_i)$ lies in $\dis \sum_{\mu=1}^m\sum_{\nu=1}^r \conj K\,z_{\mu,\nu}$. The result now follows.
\end{proof}

\subsection{The transition matrix from $\mcE_k$ to $G_k^{\oplus N}$}
Let $G_k^{\oplus N}$ be the $t$-module given by direct sum, i.e., $$G_k^{\oplus N}(t)=\begin{pmatrix}
G_k(t) & 0 & \cdots & 0\\
0 & G_k(t) & \ddots & \vdots\\
\vdots & \ddots & \ddots & 0\\
0 & \cdots & 0 & G_k(t)
\end{pmatrix},$$
an $N\times N$ block diagonal matrix whose diagonal blocks are $G_k(t)$ and whose off-diagonal blocks are zero. 

We further define the map $$p_{k,N}: \Lie_{G_k^{\oplus N}}(\CC_\infty) = \bigoplus_{\alpha = 1}^N\Lie_{G_k}(\CC_\infty)\to \CC_\infty^{rN}$$ by $$\bz = \begin{pmatrix}
	\bz_1\\
	\vdots\\
	\bz_N
\end{pmatrix}\mapsto \begin{pmatrix}
p_k(\bz_1)\\
\vdots\\
p_k(\bz_N)
\end{pmatrix}.$$

For the comparison of $\mcE_k$ with $G_k^{\oplus N}$, we use the following permutation of coordinates. We first consider the index set $$\mcJ_k:=\{(\ell,j)\mid 1\leq \ell\leq r,\ 0\leq j\leq k-1\}\cup\{(1,k)\}.$$ We also denote $\gamma_{jr+\ell} = (\ell,j)$, and write $$\nu_k:\mcJ_k\stackrel{\sim}{\longrightarrow}\{1,\ldots,rk+1\}$$ with $\nu_k(\ell,j)=jr+\ell$ for all $1\leq \ell\leq r,\ 0\leq j\leq k-1$ and $\nu_k(1,k)=rk+1$. Thus $\mcJ_k$ is the index set of the coordinates of $\Lie(G_k)$ in the above basis.  Using such a notation will be convenient to keep track of coordinates in the proof of Lemma \ref{lem:perG} and Proposition \ref{prop:Gk-tran}.

We define the permutation matrix $$\mathfrak S_k:\bigoplus_{i=1}^N\Lie_{G_k}(\CC_\infty)\longrightarrow\bigoplus_{\gamma\in\mcJ_k}\CC_\infty^N$$ by $$\mathfrak S_k\begin{pmatrix}
	\bx_1\\
	\bx_2\\
	\vdots\\
	\bx_N
\end{pmatrix}=\begin{pmatrix}
	\bx[\gamma_1]\\
	\bx[\gamma_2]\\
	\vdots\\
	\bx[\gamma_{rk+1}]
\end{pmatrix},
$$ where $\bx_i=(x_{i,\gamma_1},\cdots,x_{i,\gamma_{rk+1}})^\top\in \Lie_{G_k}(\CC_\infty)$ for each $1\leq i\leq N$ and $\bx[\gamma] = (x_{1,\gamma},\cdots,x_{N,\gamma})^\top\in \CC_\infty^N$ for each $\gamma\in \mcJ_k$.

Equivalently, if we put $m_k:=rk+1$,$d_k:=Nm_k$ and let $$
\be_1=(1,0,\ldots,0),\quad
\be_2=(0,1,\ldots,0),\quad
\ldots,\quad
\be_{m_k}=(0,\ldots,0,1)$$ be the standard basis for $\Mat_{1\times m_k}(\CC_\infty)$, then we have
$$ \mfS_k=\begin{pmatrix}
\be_1 & 0 & \cdots & 0\\
0 & \be_1 & \cdots & 0\\
\vdots & \vdots & \ddots & \vdots\\
0 & 0 & \cdots & \be_1\\
\be_2 & 0 & \cdots & 0\\
0 & \be_2 & \cdots & 0\\
\vdots & \vdots & \ddots & \vdots\\
0 & 0 & \cdots & \be_2\\
\vdots & \vdots & \ddots & \vdots\\
\be_{m_k} & 0 & \cdots & 0\\
0 & \be_{m_k} & \cdots & 0\\
\vdots & \vdots & \ddots & \vdots\\
0 & 0 & \cdots & \be_{m_k}
\end{pmatrix}
\in \GL_{d_k}(\FF_q),$$
where $0$ denotes the zero row vector in $\Mat_{1\times m_k}(\CC_\infty)$ and each row has $N$ blocks of row vectors in $\Mat_{1\times m_k}(\CC_\infty)$ and hence $\mfS_k$ has size $d_k\times d_k$.

Let $P=P(\vec{\mathbf{u}})=(M(\vec{\mathbf{u}})^\top)^{-1}\in \GL_N(\conj K)$ be as in Theorem \ref{thm:AT} so that $\mcE_t=P^{-1}\varphi_t^{\oplus N}P$. For $1\leq \ell\leq r$, put $$P_\ell:=P^{(\ell-1)}.$$

Let $\Pi_k$ be the following $(rk+1)\times(rk+1)$ block matrix with $N\times N$ blocks given by $$\Pi_k:=\operatorname{diag}(P_1,\cdots, P_r, \cdots, P_1,\cdots, P_r, P_1)\in \GL_{d_k}(K^\sep), $$ where $P_1,\ldots,P_r$ repeat $k$ times. Finally, define $$\widetilde\Pi_k:=\mfS_k^{-1}\Pi_k.$$

\begin{eg}\label{eg:Sk}
Assume $r=2$, $N=2$, and $k=1$. Then, $$\mcJ_1=\{\gamma_1=(1,0),\gamma_2=(2,0),\gamma_3=(1,1)\}.$$

Write $$P_1=P= \begin{pmatrix}
a&b\\
c&d
\end{pmatrix}, \qquad P_2=P^{(1)}=
\begin{pmatrix}
a^q&b^q\\
c^q&d^q
\end{pmatrix}.$$

We have
$$\mfS_1 = \begin{pmatrix}
1 & 0 & 0 & 0 & 0 & 0\\
0 & 0 & 0 & 1 & 0 & 0\\
0 & 1 & 0 & 0 & 0 & 0\\
0 & 0 & 0 & 0 & 1 & 0\\
0 & 0& 1 & 0 & 0 & 0\\
0 & 0 & 0 & 0 & 0 &  1
\end{pmatrix}\qquad\textnormal{ and }\qquad \Pi_1=
\begin{pmatrix}
a&b&0&0&0&0\\
c&d&0&0&0&0\\
0&0&a^q&b^q&0&0\\
0&0&c^q&d^q&0&0\\
0&0&0&0&a&b\\
0&0&0&0&c&d
\end{pmatrix}.$$

Thus, we see that $$\begin{pmatrix}
	\bx_1\\
	\bx_2
\end{pmatrix} = (x_{1,\gamma_1},x_{1,\gamma_2},x_{1,\gamma_3}, x_{2,\gamma_1},x_{2,\gamma_2},x_{2,\gamma_3})^\top$$ and $$ \begin{pmatrix}
\bx[\gamma_1]\\
\bx[\gamma_2]\\
\bx[\gamma_3]
\end{pmatrix} = (x_{\gamma_1,1},x_{\gamma_1,2},x_{\gamma_2,1},x_{\gamma_2,2},x_{\gamma_3,1},x_{\gamma_3,2})^\top$$ satisfy $$\mfS_1 \begin{pmatrix}
\bx_1\\
\bx_2
\end{pmatrix} =\begin{pmatrix}
\bx[\gamma_1]\\
\bx[\gamma_2]\\
\bx[\gamma_3]
\end{pmatrix}.$$

Furthermore, we have $$\widetilde\Pi_1=\mfS_1^{-1}\Pi_1=
\begin{pmatrix}
a&b&0&0&0&0\\
0&0&a^q&b^q&0&0\\
0&0&0&0&a&b\\
c&d&0&0&0&0\\
0&0&c^q&d^q&0&0\\
0&0&0&0&c&d
\end{pmatrix}.$$

\end{eg}
\begin{lem}\label{lem:per}
	Let $P\in \Mat_m(\FF_q)$ be a permutation matrix. Let $\sigma\in S_m$ be the permutation corresponding to $P$ such that $P\be_i=\be_{\sigma(i)}$ for all $1\leq i\leq m$, where $\be_1,\ldots,\be_m$ are the standard column vectors in $\Mat_{m\times1}(\FF_q)$. Let $M\in \Mat_m(\CC_\infty)$ and put $C=PMP^{-1}$. Then, $$C_{\sigma(i),\sigma(j)}=M_{i,j}$$
	for all $1\leq i,j\leq m$.
\end{lem}
\begin{proof}
Since $P\be_i=\be_{\sigma(i)}$ and $P^\top = P^{-1}$, we have $P^\top\be_{\sigma(i)} = P^{-1}\be_{\sigma(i)}=\be_i$ for every $1\leq i\leq m$. Hence, for all $1\leq i,j\leq m$, $$C_{\sigma(i),\sigma(j)}=\be_{\sigma(i)}^\top C\be_{\sigma(j)}=\be_{\sigma(i)}^\top PMP^{-1}\be_{\sigma(j)}=\be_i^\top M\be_j=M_{i,j}.$$
This proves the assertion.
\end{proof}

Let $M\in\Mat_{d_k}(\CC_\infty)$ and $m_k=|\mcJ_k|=rk+1$. If we regard $M$
as an $m_k\times m_k$ block matrix whose blocks lie in
$\Mat_N(\CC_\infty)$, then for $\alpha,\beta\in\mcJ_k$ we write
$M[\alpha,\beta]$ for the $(\nu_k(\alpha),\nu_k(\beta))$-block of $M$, and
$M_{(\alpha,i),(\beta,j)}$ denotes the $(i,j)$-entry of $M[\alpha,\beta]$.
If we regard $M$ as an $N\times N$ block matrix whose blocks lie in
$\Mat_{m_k}(\CC_\infty)$, then $M_{(i,\alpha),(j,\beta)}$ denotes the
$(\nu_k(\alpha),\nu_k(\beta))$-entry of the $(i,j)$-block.

\begin{lem}\label{lem:perG}
	Let $\mcU_k$ be the matrix defined in \eqref{twist_k} and $\mcB_k$ be the block matrix defined as follows
	$$\mcB_k=\begin{pmatrix}
		0&0& \cdots&\cdots&\cdots&\cdots & \cdots & 0\\
		\vdots& \vdots & & & & & & \vdots\\
		0&0 & \cdots&\cdots&\cdots&\cdots &\cdots &0\\
		I_N&0& 0&\cdots&\cdots&\cdots & \cdots & 0\\
		0&I_N&\ddots&&&&&\vdots\\
		\vdots&\vdots&\ddots&\ddots & & &&\vdots\\
		0&0&\cdots&I_N& 0 &0&\cdots&0\\
		a_1I_N&a_2I_N& \cdots &a_{r-1}I_N&a_rI_N&0&\cdots&0
	\end{pmatrix}\!
	\begin{aligned}
		&\left. \vphantom{\begin{matrix} 0 \\ \vdots\\ 0 \end{matrix}} \right\} rk+1-r \\
		&\left. \vphantom{\begin{matrix} 0 \\  \vdots\\ \vdots\\  0\\ 0 \end{matrix}} \right\} r
	\end{aligned}
	$$ Then, $$\mfS_k\cdot G_k^{\oplus N}(t)\cdot \mfS_k^{-1}=\theta I_{d_k}+\mcU_k+\mcB_k\tau.$$
\end{lem}
\begin{proof}
Recall that $m_k=rk+1$ and write
$\mcJ_k=\{\gamma_1,\ldots,\gamma_{m_k}\}$ with the fixed ordering above. For $$\bx=\left(x_{1,\gamma_1},\ldots,x_{1,\gamma_{m_k}},
x_{2,\gamma_1},\ldots,x_{2,\gamma_{m_k}},
\ldots,x_{N,\gamma_1},\ldots,x_{N,\gamma_{m_k}}\right)^\top
\in \Lie(G_k^{\oplus N}),$$ the definition of $\mfS_k$ gives $$\mfS_k\bx=
\left(x_{1,\gamma_1},\ldots,x_{N,\gamma_1},
x_{1,\gamma_2},\ldots,x_{N,\gamma_2},
\ldots,x_{1,\gamma_{m_k}},\ldots,x_{N,\gamma_{m_k}}\right)^\top
\in \bigoplus_{\gamma\in\mcJ_k}\CC_\infty^N.$$
Thus $\mfS_k$ is the permutation matrix attached to the bijection $$
\sigma:\{1,\ldots,N\}\times\mcJ_k\longrightarrow\mcJ_k\times\{1,\ldots,N\},$$ $$(i,\gamma)\mapsto (\gamma,i).$$

Write $$G_k(t)=\theta I_{m_k}+N_k+B_k\tau.$$ Then, $$G_k^{\oplus N}(t)
=\theta I_{d_k}+N_k^{\oplus N}+B_k^{\oplus N}\tau$$ with rows and columns
indexed by $\{1,\ldots,N\}\times\mcJ_k$. Since $\mfS_k$ is a permutation
matrix with entries in $\FF_q$, we have $$\mfS_k(B_k^{\oplus N}\tau)\mfS_k^{-1}=\mfS_kB_k^{\oplus N}(\mfS_k^{-1})^{(1)}\tau=\mfS_kB_k^{\oplus N}\mfS_k^{-1}\tau.$$

For any matrix $M$ with rows and columns indexed by $\{1,\ldots,N\}\times\mcJ_k$, Lemma \ref{lem:per} gives
$$(\mfS_kM\mfS_k^{-1})_{(\alpha,i),(\beta,j)}
=M_{(i,\alpha),(j,\beta)}$$ for all $\alpha,\beta\in\mcJ_k,\ 1\leq i,j\leq N$.
Let $H=N_k^{\oplus N}$. Then, $$H_{(i,\alpha),(j,\beta)}=\begin{cases}
(N_k)_{\nu_k(\alpha),\nu_k(\beta)},& i=j,\\
0,& i\neq j.
\end{cases}$$
Therefore, for all $\alpha,\beta\in\mcJ_k$, $$(\mfS_kN_k^{\oplus N}\mfS_k^{-1})[\alpha,\beta]=(N_k)_{\nu_k(\alpha),\nu_k(\beta)}I_N.$$
By the definition of $N_k$, its only non-zero entries are $(N_k)_{s,s+r}=1$, where $1\leq s\leq rk+1-r$. Equivalently, under the map $\nu_k$, these entries are $$(N_k)_{\nu_k(\ell,j),\nu_k(\ell,j+1)}=1, \quad\textnormal{ for all }1\leq \ell\leq r,\ 0\leq j\leq k-2,$$
and $$(N_k)_{\nu_k(1,k-1),\nu_k(1,k)}=1.$$
Thus $\mfS_kN_k^{\oplus N}\mfS_k^{-1}$ is exactly the block matrix $\mcU_k$ defined in \eqref{twist_k}.

Applying the same argument to $B_k^{\oplus N}$ gives, for all
$\alpha,\beta\in\mcJ_k$, $$(\mfS_kB_k^{\oplus N}\mfS_k^{-1})[\alpha,\beta]=(B_k)_{\nu_k(\alpha),\nu_k(\beta)}I_N.$$

By the definition of $B_k$, the non-zero entries of $B_k$ are precisely $$\begin{cases}
	(B_k)_{\nu_k(\ell,k-1),\nu_k(\ell-1,0)}=1,\quad2\leq \ell\leq r,\\
	(B_k)_{\nu_k(1,k),\nu_k(\ell,0)}=a_\ell,\quad1\leq \ell\leq r.
\end{cases}$$

Hence, $$\mfS_kB_k^{\oplus N}\mfS_k^{-1}=\mcB_k.$$ It follows that $$\mfS_kG_k^{\oplus N}(t)\mfS_k^{-1}=\theta I_{d_k}+\mcU_k+\mcB_k\tau.$$
\end{proof}
\begin{prop}\label{prop:Gk-tran}
	Let $G_k=G_{\varphi,k}$. With the matrices $\Pi_k$ and
	$\widetilde\Pi_k$ defined above, we have
	$$\mcE_k(t)	=\widetilde\Pi_k^{-1}\cdot G_k^{\oplus N}(t)\cdot \widetilde\Pi_k.$$
	
	In particular, $$\mcE_k\otimes_K K^\sep\cong G_k^{\oplus N}\otimes_K K^\sep.$$
\end{prop}
\begin{proof}
By Lemma \ref{lem:perG}, we have
$$\mfS_k\cdot G_k^{\oplus N}(t)\cdot \mfS_k^{-1}=\theta I_{d_k}+\mcU_k+\mcB_k\tau.$$

It is easy to see that if $\dis \bx = (\bx[\gamma_1],\cdots, \bx[\gamma_{rk+1}])^\top\in \bigoplus_{\gamma\in\mcJ_k}\CC_\infty^N = \CC_\infty^{d_k}$, then  $$(H\bx)[\alpha]=\sum_{\beta\in\mcJ_k}H[\alpha,\beta]\bx[\beta].$$
Since $\Pi_k$ is block diagonal,  if $\gamma = (\ell,j)\in \mcJ_k$, put $$\Pi_k[\gamma]:=\Pi_k[\gamma,\gamma] = P_\ell.$$

Thus, we must have $$(\Pi_k^{-1}\mcU_k\Pi_k)[\alpha,\beta]=\Pi_k[\alpha]^{-1}\mcU_k[\alpha,\beta]\Pi_k[\beta]$$ as $\Pi_k$ is block diagonal. If this block is non-zero, then $\mcU_k[\alpha,\beta]\neq 0$, i.e. $\alpha = (\ell,j)$ and $\beta = (\ell,j+1)$ with $1\leq jr+\ell\leq rk+1-r$. Equivalently, either $\alpha=(\ell,j)$, $\beta=(\ell,j+1)$ with $1\leq \ell\leq r,\ 0\leq j\leq k-2$ or $\alpha=(1,k-1)$, $\beta=(1,k)$. In both cases, $\Pi_k[\alpha]=\Pi_k[\beta] = P_\ell$. Hence, $(\Pi_k^{-1}\mcU_k\Pi_k)[\alpha,\beta]=I_N$ (resp. $0$) if and only if $\mcU_k[\alpha,\beta]=I_N$ (resp. $0$). Therefore, we have 
$$\Pi_k^{-1}\mcU_k\Pi_k=\mcU_k.$$

Similarly, we have $$(\Pi_k^{-1}\mcB_k\Pi_k^{(1)})[\alpha,\beta]
=\Pi_k[\alpha]^{-1}\mcB_k[\alpha,\beta]\Pi_k[\beta]^{(1)}.$$
By the definition of $\mcB_k$, we have $$\mcB_k[\alpha,\beta] = \begin{cases}
	I_N,\textnormal{ if }\alpha = (\ell,k-1), \beta = (\ell-1,0), 2\leq \ell\leq r\\
	a_\ell I_N, \textnormal{ if }\alpha = (1,k), \beta = (\ell,0), 1\leq \ell\leq r\\
	0, \textnormal{ otherwise}.
\end{cases}$$ Thus, in the first case, we have $$(\Pi_k^{-1}\mcB_k\Pi_k^{(1)})[\alpha,\beta]=P_\ell^{-1}I_NP_{\ell-1}^{(1)}=(P^{(\ell-1)})^{-1}P^{(\ell-1)}=I_N.$$ In the second case, $$(\Pi_k^{-1}\mcB_k\Pi_k^{(1)})[\alpha,\beta]=P_1^{-1}(a_\ell I_N)P_\ell^{(1)}=a_\ell P^{-1}P^{(\ell)}=A_\ell,$$ by Theorem \ref{thm:AT}. Thus, we have $$\Pi_k^{-1}\mcB_k\Pi_k^{(1)}=\mcA_k.$$
Combining the discussion above gives $$\Pi_k^{-1}\mfS_k\cdot G_k^{\oplus N}(t)\cdot \mfS_k^{-1}\Pi_k=\Pi_k^{-1}(\theta I_{d_k}+\mcU_k+\mcB_k\tau)\Pi_k=\theta I_{d_k}+\mcU_k+\mcA_k\tau=\mcE_k(t).$$

Since $\widetilde\Pi_k=\mathfrak S_k^{-1}\Pi_k$, this is precisely
$$\widetilde\Pi_k^{-1}\cdot G_k^{\oplus N}(t)\cdot \widetilde\Pi_k=\mcE_k(t).$$
The last assertion follows immediately.
\end{proof}

\subsection{The proof of Theorem \ref{thm:int}}

Since the ground field is now $K$, the ring of integers is $A$. We first choose integral models for the Artin twists.

\begin{lem}\label{lem:den}
	There exists $b\in A-\{0\}$ such that the following hold for all $k\geq 1$:
	\begin{enumerate}
		\item the $t$-modules $\widehat{\mcE}:= b^{-1}\mcE b$ and $\widehat{\mcE}_{k}:=b^{-1}\mcE_{k}b$,  are defined over $A$;
		\item $\widehat{\mcE}_k = \widehat{\mcE}\otimes C^{\otimes k}$, i.e. $\MM(\widehat{\mcE}_k)\cong \MM(\widehat{\mcE})\otimes \MM(C)^{\otimes k}$;
		\item $\partial\widehat{\mcE}_{k}(t)-\theta I_{d_k}=\mcU_k$, where $\mcU_k$ is the nilpotent matrix in \eqref{twist_k};
		\item $\exp_{\widehat \mcE} = b^{-1}\exp_{\mcE}b$ and $\exp_{\widehat{\mcE}_k}=b^{-1}\exp_{\mcE_k}b$.
	\end{enumerate}
\end{lem}

\begin{proof}
	By \eqref{twist_k}, we have $\mcE_{k}(t)=\theta I_{d_k}+\mcU_k+\mcA_k\tau$, where the entries of $\mcU_k$ lie in $A$ and the entries of $\mcA_k$ lie in $K$. Choose $b\in A-\{0\}$ such that $b^{q-1}A_i$ has entries in $A$ for all $1\leq i\leq r$. Then $b^{q-1}\mcA_k$ has entries in $A$. Since $b$ commutes with $\theta I_{d_k}+\mcU_k$, and
	$$b^{-1}(\mcA_k\tau)b=b^{-1}\mcA_k b^{(1)}\tau=b^{q-1}\mcA_k\tau,$$
	we have $$\widehat{\mcE}_{k}(t)=\theta I_{d_k}+\mcU_k+b^{q-1}\mcA_k\tau\in \Mat_{d_k}(A)[\tau].$$
	Similarly, $\widehat{\mcE}(t)\in \Mat_{N}(A)[\tau]$. Thus (1) and (3) hold, and (2) follows because $b\in K$. The identity for the exponential map follows from uniqueness of the exponential map.
\end{proof}
Consider the projection map $$\pi_k:\Lie_{\widehat{\mcE}_k}(K_\infty)\longrightarrow K_\infty^{rN}$$ given by $$(y_1,\ldots,y_{rk+1})^\top\longmapsto(y_{r(k-1)+2},\ldots,y_{rk+1})^\top.$$ It has kernel $\mcU_k\Lie_{\widehat{\mcE}_k}(K_\infty)$. Indeed, by Lemma \ref{lem:den}, $\partial\widehat{\mcE}_{k}(t)-\theta I_{d_k}=\mcU_k$. Recall that $d_k = N(rk+1)$ and $\Lie_{\widehat{\mcE}_k}(K_\infty) = K_\infty^{d_k}$. For
$$\by=(y_1,\ldots,y_{rk+1})^\top,\qquad y_i\in K_\infty^N,$$ we have
$$\mcU_k\by=(y_{r+1},\ldots,y_{rk+1},0,\ldots,0)^\top.$$

Thus, we have the following result.
\begin{lem}\label{lem:Wk}
Let $k\geq1$. Then the quotient
$$W_{\widehat{\mcE}_k}(K_\infty)=\Lie_{\widehat{\mcE}_k}(K_\infty)/(\partial\widehat{\mcE}_{k}(t)-\theta I_{d_k})\Lie_{\widehat{\mcE}_k}(K_\infty)$$
can be identified with the last $r$ blocks of $\Lie(\widehat{\mcE}_k)(K_\infty)$ via $\pi_k$. In particular, $W_{\widehat{\mcE}_k}(K_\infty)\cong K_\infty^{rN}$ and $W_{\widehat{\mcE}_k}(A)\cong A^{rN}$.
\end{lem}

\begin{lem}[cf. {\cite[Theorem 4.1]{GN24}}]\label{lem:Z}
	There exists a $K_\infty$-subspace $\mcZ\subseteq \Lie_{\widehat{\mcE}_k}(K_\infty)$ such that
	\begin{enumerate}
		\item the quotient map $\pi_k:\Lie_{\widehat{\mcE}_k}(K_\infty)\to W_{\widehat{\mcE}_k}(K_\infty)$ induces an isomorphism $\mcZ \simeq W_{\widehat{\mcE}_k}(K_\infty)$;
		\item $U(\widehat{\mcE}_k/A)\subseteq \Lie_{\widehat{\mcE}_k}(K)+\mcZ$;
		\item $U(\widehat{\mcE}_k/A)\cap \mcZ$ and $\Lie_{\widehat{\mcE}_k}(A)\cap \mcZ$ are $\bA$-lattices in $\mcZ$ and
		$$L(\widehat{\mcE}_k^\vee,0) = c_k\cdot [\Lie_{\widehat{\mcE}_k}(A)\cap \mcZ:U(\widehat{\mcE}_k/A)\cap \mcZ]_A$$ for some $c_k\in K^\times$.
	\end{enumerate}
\end{lem}

\begin{proof}
	This follows from \cite[Theorem 4.4, Corollary 4.5]{ANDTR20}.
\end{proof}

For the rest of this subsection, write $W_k$ for $W_{\widehat{\mcE}_k}$ and $\Lie_k$ for $\Lie_{\widehat{\mcE}_k}$. Set $\delta=N(rk+1-r)$. By Lemma \ref{lem:Z}, we can choose an $A$-basis $e_1,\ldots,e_{N(rk+1)}$ of $\Lie_k(A)$ such that $e_{\delta+1},\ldots,e_{\delta+rN}$ is an $A$-basis of $\Lie_k(A)\cap\mcZ$, and $\pi_k(e_{\delta+1}),\ldots,\pi_k(e_{\delta+rN})$ is the standard $A$-basis of $W_k(A)$. Choose an $A$-basis $g_1,\ldots,g_{rN}$ of $U(\widehat{\mcE}_k/A)\cap\mcZ$. 

Let $h:\mcZ\to\mcZ$ be the $K_\infty$-linear map given by  $$h(e_{\delta+j})=g_j,$$ for all $1\leq j\leq rN$.
Let $\widehat h:W_k(K_\infty)\to W_k(K_\infty)$ be the $K_\infty$-linear map induced by $h$ through $\pi_k$. Write $$\pi_k(g_j)=\sum_{i=1}^{rN}g_{ij}\pi_k(e_{\delta+i})$$ for some $g_{ij}\in K_\infty$ and put $$M(\widehat h):=(g_{ij})_{1\leq i,j\leq rN}.$$

Then, it is immediate to see the following.
\begin{cor}\label{cor:reg}
	We have $\det M(\widehat h)\neq0$ and $$L(\widehat{\mcE}_k^\vee,0)=c_k\det M(\widehat h)$$ for some $c_k\in K^\times$.
\end{cor}
\begin{proof}
	This simply follows from Lemma \ref{lem:Z}(3) as $[\Lie_{\widehat{\mcE}_k}(A)\cap \mcZ:U(\widehat{\mcE}_k/A)\cap \mcZ]_A = \det M(\widehat{h})$.
\end{proof}
Recall that $\widetilde\Pi_k=\mathfrak S_k^{-1}\Pi_k$ and $P_h:=P^{(h-1)}$ for all $1\leq h\leq r$, where $P\in \GL_N(K^\sep)$ is the matrix in Theorem \ref{thm:AT}. We further define $P_{r+1}:=P_1$, and $$B_{k}:=\operatorname{diag}(bP_2,\ldots,bP_r,bP_1)\in \GL_{rN}(\conj K).$$

For $1\leq j\leq rN$, define $\by_j:=g_j\in \Lie_{k}(K_\infty)$ and $\bz_j:=b\widetilde\Pi_k\by_j\in \Lie_{G_k^{\oplus N}}(\CC_\infty)$.
Write $$\bz_j=\begin{pmatrix}
\bz_{j,1}\\
\vdots\\
\bz_{j,N}
\end{pmatrix},
\qquad
\bz_{j,\alpha}\in \Lie_{G_k}(\CC_\infty).$$

Define $$Y:=[\pi_k(\by_1),\ldots,\pi_k(\by_{rN})]\in \Mat_{rN}(\CC_\infty),$$ and $$Z:=[p_{k,N}(\bz_1),\ldots,p_{k,N}(\bz_{rN})]\in \Mat_{rN}(\CC_\infty).$$
Since $\by_j=g_j=h(e_{\delta+j})$, we have $$\pi_k(\by_j)=\widehat h\bigl(\pi_k(e_{\delta+j})\bigr).$$

By the definition of the coefficients $g_{ij}$, the $j$-th column of $M(\widehat h)$ is precisely the coordinate vector of $\pi_k(\by_j)$ with respect to the basis $\pi_k(e_{\delta+1}),\ldots,\pi_k(e_{\delta+rN})$ of $W_k(K_\infty)$. Hence, we have $Y=M(\widehat h)$ as $\pi_k(e_{\delta+1}),\ldots,\pi_k(e_{\delta+rN})$ is the standard $A$-basis of $W_k(A)$.

Define the permutation matrix $$\mathfrak R_{r,N}:\bigoplus_{\alpha=1}^{N}\CC_\infty^r\longrightarrow\bigoplus_{h=1}^r\CC_\infty^N$$ by $$
\mathfrak R_{r,N}\begin{pmatrix}
\bw_1\\
\vdots\\
\bw_N
\end{pmatrix}=\begin{pmatrix}
\bw[1]\\
\vdots\\
\bw[r]
\end{pmatrix},$$ where $\bw_\alpha=(w_{\alpha,1},\ldots,w_{\alpha,r})^\top$ for $1\leq \alpha\leq N$ and $\bw[h]=(w_{h,1},\ldots,w_{h,N})^\top$ for $1\leq h\leq r$.

\begin{lem}\label{lem:coor}
Let $\by\in \Lie_{\widehat{\mcE}_k}(\CC_\infty)$ and put $$\bz:=b\widetilde\Pi_k\by\in \Lie_{G_k^{\oplus N}}(\CC_\infty).$$
Then, $$\mathfrak R_{r,N}p_{k,N}(\bz)=B_{k}\pi_k(\by).$$
\end{lem}

\begin{proof}
Write $\by$ in $\mcJ_k$-blocks as $$\by=\begin{pmatrix}
\by[\gamma_1]\\
\vdots\\
\by[\gamma_{rk+1}]
\end{pmatrix},
\qquad
\by[\gamma_i]\in \CC_\infty^N.$$
By the definitions of $\widetilde\Pi_k$ and $\mathfrak S_k$, we have $$\mathfrak S_k\bz=\mathfrak S_k b\widetilde\Pi_k\by=b\Pi_k\by.$$
Let $$q_{k,N}:\bigoplus_{\gamma\in\mcJ_k}\CC_\infty^N\longrightarrow
\bigoplus_{h=1}^r\CC_\infty^N$$ be the projection onto the last $r$ blocks. By the definitions of $p_{k,N}$, $\mathfrak S_k$, and $\mathfrak R_{r,N}$, we have $$\mathfrak R_{r,N}p_{k,N}(\bz)=q_{k,N}(\mathfrak S_k\bz) =q_{k,N}(b\Pi_k\by).$$

By the definition of $\Pi_k$, the latter is simply $$\begin{pmatrix}
	bP_2 & & &\\
	& \ddots & &\\
	& & bP_r &\\
	&&&bP_1
\end{pmatrix}\begin{pmatrix}
\by[\gamma_{r(k-1)+2}]\\
\by[\gamma_{r(k-1)+3}]\\
\vdots\\
\by[\gamma_{rk+1}]
\end{pmatrix}.$$
Thus, $$\mathfrak R_{r,N}p_{k,N}(\bz)=B_{k}\pi_k(\by).$$
\end{proof}

\begin{lem}\label{lem:detZ}
We have $$\mathfrak R_{r,N}Z=B_{k}Y=B_{k}M(\widehat h).$$
In particular, $$\det(Z)=\det(\mathfrak R_{r,N})^{-1}\det(B_{k})\det M(\widehat h)\neq0.$$
\end{lem}

\begin{proof}
Apply Lemma \ref{lem:coor} to $\by_1,\ldots,\by_{rN}$, then the result follows. The second assertion follows because $\mathfrak R_{r,N}$, $B_{k}$, and $M(\widehat h)$ are invertible.
\end{proof}

\begin{lem}\label{lem:algpt}
For all $j$ and $\alpha$, one has $$\exp_{G_k}(\bz_{j,\alpha})\in G_k(\conj K).$$
\end{lem}

\begin{proof}
Let $Q_k=b\widetilde\Pi_k$. From Proposition \ref{prop:Gk-tran} and the definition of $\widehat{\mcE}_k$, we have $$\widehat{\mcE}_{k}(t)=Q_k^{-1}(G_k^{\oplus N}(t))Q_k.$$ 

Thus, we have $$\exp_{\widehat{\mcE}_k} = Q_k^{-1}\exp_{G_k^{\oplus N}} Q_k.$$

Since $g_j\in U(\widehat{\mcE}_k/A)\cap\mcZ$, we have $$\exp_{\widehat{\mcE}_k}(g_j)\in \widehat{\mcE}_k(A).$$

Thus, $$\exp_{G_k^{\oplus N}}(\bz_j)=Q_k\exp_{\widehat{\mcE}_k}(g_j)
\in (G_k^{\oplus N})(\conj K).$$

Taking the $\alpha$-th component proves the assertion.
\end{proof}

\begin{prop}\label{prop:det-tran}
The determinant $\det(Z)$ is transcendental over $\conj K$.
\end{prop}

\begin{proof}
Put $$\mcZ_k:=\{\bz_{j,\alpha}\mid 1\leq j\leq rN,\ 1\leq \alpha\leq N\}.$$
By Lemma \ref{lem:algpt}, every element of $\mathcal Z_k$ is a logarithm of an algebraic point on $G_k$. Choose a maximal subset
$\{\bw_1,\ldots,\bw_m\}\subseteq \mcZ_k$, which is linearly independent over $F_k=\Frac(\End_{\conj K}(G_k))$. Write $$p_k(\bw_\alpha)=(w_{\alpha,1},\ldots,w_{\alpha,r})^\top.$$

By Theorem \ref{thm:Gk-ai}, all $w_{\alpha,h}$ are algebraically independent over $\conj K$.  Note that every entry of $Z$ is a coordiante of $p_k(\bz_{j,\alpha})$ for some $1\leq j\leq rN$ and $1\leqq \alpha\leq N$. By Lemma \ref{lem:Gk-lin}, it is a $\conj K$-linear form in the variables $w_{\alpha,h}$. Hence there is a matrix $\mathsf Z(X)$ whose entries are homogeneous linear forms over $\conj K$ such that $$Z=\mathsf Z(w_{\alpha,h}).$$

Set $H(X):=\det \mathsf Z(X)$. By Lemma \ref{lem:detZ}, $$H(w_{\alpha,h})=\det(Z)\neq0.$$

Thus $H$ is a non-zero homogeneous polynomial of positive degree over $\conj K$. Argue by contradiction, we assume that $\det(Z)$ is algebraic over $K$, say $\det(Z)=c\in\conj K$. Then $$H(w_{\alpha,h})-c=0$$ gives a non-trivial algebraic relation among $\{w_{\alpha,h}\}$ over $\conj K$, contradicting the algebraic independence of the $\{w_{\alpha,h}\}$ over $\conj K$.
\end{proof}

\begin{thm}\label{thm:main}
Let $\varphi$ be a Drinfeld module over $K$, and let $$\rho:G_K\longrightarrow\GL_n(\conj\FF_q)$$
be an Artin representation. Then, for every integer $k\geq0$, the special value $L(\varphi^\vee,\rho,k)$ is transcendental over $\conj K$.
\end{thm}

\begin{proof}
For $k=0$, this follows from \cite[Theorem 5.2]{Ye26}. Now suppose $k\geq1$. By Lemma \ref{cor:reg} and Lemma \ref{lem:detZ}, $$L(\widehat{\mcE}_k^\vee,0)=c_k\det M(\widehat h)=c_k\det(B_{k})^{-1}\det(\mfR_{r,N})\det(Z)$$
for some $c_k\in K^\times$. Since $\det(B_{k})$ and $\det(\mfR_{r,N})$ are both algebraic over $K$ by the definition of $B_k$ and $\mfR_{r,N}$, Proposition \ref{prop:det-tran} implies that $L(\widehat{\mcE}_k^\vee,0)$ is transcendental over $\conj K$. Since $\widehat{\mcE}_k$ and $\mcE_k$ are isomorphic over $K$, we see that $L(\mcE_k^\vee,0)$ is transcendental over $\conj K$.

For a finite set of places $S$ containing the bad places, Corollary \ref{cor:Ek-L} gives
$$L_S(\mcE_k^\vee,0)=
N_{K_\infty(\rho)/K_\infty}\bigl(L_S(\varphi^\vee,\rho,k)\bigr).$$
Note that the local factors at each prime in $S$ is non-zero algebraic. If $L(\varphi^\vee,\rho,k)$ were algebraic over $\conj K$, then the right hand side would also be algebraic over $\conj K$, and hence so would $L(\mcE_k^\vee,0)$, a contradiction.
\end{proof}

\begin{cor}\label{cor:car}
Let $\rho:G_K\to \GL_n(\conj\FF_q)$ be an Artin representation. Then, for
every integer $k\geq1$, the special $L$-value $L(\rho,k)$ is transcendental
over $\conj K$.
\end{cor}

\begin{proof}
Take $\varphi=C$, the Carlitz module, in Theorem \ref{thm:main}. Then $L(C^\vee,\rho,k-1)$ is transcendental over $\conj K$ for every $k\geq1$. By Lemma \ref{lem:carL}, $L(C^\vee,\rho,k-1)$ and $L(\rho,k)$ differ at most by a finite product of non-zero algebraic local factors. Hence $L(\rho,k)$ is also transcendental over $\conj K$.
\end{proof}

\subsection{Further directions over finite extensions and obstructions}
It is natural to ask whether Conjecture \ref{cjt:main} remains true when
$E/K$ is a non-trivial finite separable extension, $\varphi$ is a
Drinfeld module over $E$, and $\rho$ is an Artin representation of
$G_E$. The proof in this paper suggests a possible strategy. It also
shows where a new input is needed.

First, the embeddings of $E$ over $K$ change the shape of the problem. When we embed $E$ into $\CC_\infty$ and compute the logarithms, the object attached to $\varphi$ decomposes according to the $K$-embeddings $\sigma:E\hookrightarrow \CC_\infty$. Its pieces are attached to the conjugate Drinfeld modules $\varphi^\sigma$. After tensoring with $C^{\otimes k}$, one expects the corresponding $t$-module to decompose as a direct sum of $t$-modules of the form $G_{\varphi^\sigma,k}$, with multiplicities determined by the Artin representation. Thus one should be able to choose transition matrices for the different conjugates and put them together into one block matrix. In this way, one should still get an analogue of Lemma \ref{lem:detZ}. More precisely, after a suitable permutation of coordinates, the matrix $Z_E$ containing the relevant tractable coordinates should satisfy an identity of the form $$\mfR_E Z_E=B_E M(\widehat h_E),$$ where $\mfR_E$ and $B_E$ are an block matrices with algebraic entries and $M(\widehat h_E)$ is the matrix coming from the regulator. Hence one would also get an identity of the form $$\det(Z_E)=\det(\mathfrak R_E)^{-1}\det(B_E)\det M(\widehat h_E).$$
So the regulator can still be reduced to a determinant whose entries are
algebraic linear forms in tractable coordinates of logarithms.

The real difficulty is the last step of the transcendence argument. When $E=K$, all entries of $Z$ come from logarithms on the single $t$-module $G_{\varphi,k}$, and Theorem \ref{thm:Gk-ai} applies. For a general extension $E/K$, the entries of $Z_E$ come from logarithms on several different $t$-modules $G_{\varphi^{\sigma_1},k},\ldots,G_{\varphi^{\sigma_s},k}$. Equivalently, before tensoring with the Carlitz power, they come from logarithms attached to several conjugate Drinfeld modules $\varphi^{\sigma_i}$. These Drinfeld modules may be non-isogenous. This is an essential obstruction. The algebraic independence theorems of Chang-Papanikolas (resp. Gezmi{\c s}-Namoijam) only control logarithms on one Drinfeld module $\varphi$ (resp. $t$-module $G_{\varphi,k}$). It does not tell us anything about the algebraic relations among logarithms when they are coming from several different non-isogenous Drinfeld modules or $t$-modules.

Thus a proof over a general finite extension would require a product version of the algebraic independence theorem, for logarithms coming from several Drinfeld modules or several associated $t$-modules at the same time. More precisely, one would expect a result for products $G_{\varphi_1,k}^{\oplus n_1}\oplus\cdots\oplus G_{\varphi_s,k}^{\oplus n_s}$, where $\varphi_1,\ldots,\varphi_s$ are pairwise non-isogenous Drinfeld modules. Under the expected linear independence assumptions over the endomorphism rings of these factors, the theorem should imply algebraic independence of all tractable coordinates of the relevant logarithms. Therefore, the next step is to study several non-isogenous Drinfeld modules, or the specific $t$-modules $G_{\varphi_i,k}$, at the same time, and to describe all algebraic relations among their periods and logarithms. This is the main missing part for extending the present method from $K$ to finite separable extensions of $K$.

\bibliographystyle{alpha}
\bibliography{Bibliography/FField}

@article {anderson1986t,
    AUTHOR = {Anderson, Greg W.},
     TITLE = {{$t$}-motives},
   JOURNAL = {Duke Math. J.},
  FJOURNAL = {Duke Mathematical Journal},
    VOLUME = {53},
      YEAR = {1986},
    NUMBER = {2},
     PAGES = {457--502},
      ISSN = {0012-7094,1547-7398},
   MRCLASS = {11F67 (11G05 11R58 14K05)},
  MRNUMBER = {850546},
MRREVIEWER = {David\ Goss},
       DOI = {10.1215/S0012-7094-86-05328-7},
       URL = {https://doi.org/10.1215/S0012-7094-86-05328-7},
}

@article {angles2022class,
    AUTHOR = {Angl\`es, Bruno and Ngo Dac, Tuan and Tavares Ribeiro, Floric},
     TITLE = {A class formula for admissible {A}nderson modules},
   JOURNAL = {Invent. Math.},
  FJOURNAL = {Inventiones Mathematicae},
    VOLUME = {229},
      YEAR = {2022},
    NUMBER = {2},
     PAGES = {563--606},
      ISSN = {0020-9910,1432-1297},
   MRCLASS = {11G09 (11M38 11R58)},
  MRNUMBER = {4448991},
MRREVIEWER = {David\ Tweedle},
       DOI = {10.1007/s00222-022-01110-3},
       URL = {https://doi.org/10.1007/s00222-022-01110-3},
}

@article {fang2015special,
    AUTHOR = {Fang, Jiangxue},
     TITLE = {Special {$L$}-values of abelian {$t$}-modules},
   JOURNAL = {J. Number Theory},
  FJOURNAL = {Journal of Number Theory},
    VOLUME = {147},
      YEAR = {2015},
     PAGES = {300--325},
      ISSN = {0022-314X,1096-1658},
   MRCLASS = {11G09 (11R58 13C99 14G10)},
  MRNUMBER = {3276327},
MRREVIEWER = {Liang-Chung\ Hsia},
       DOI = {10.1016/j.jnt.2014.07.012},
       URL = {https://doi.org/10.1016/j.jnt.2014.07.012},
}

@book {Goss,
    AUTHOR = {Goss, David},
     TITLE = {Basic structures of function field arithmetic},
    SERIES = {Ergebnisse der Mathematik und ihrer Grenzgebiete (3) [Results
              in Mathematics and Related Areas (3)]},
    VOLUME = {35},
 PUBLISHER = {Springer-Verlag, Berlin},
      YEAR = {1996},
     PAGES = {xiv+422},
      ISBN = {3-540-61087-1},
   MRCLASS = {11G09 (11L05 11R58)},
  MRNUMBER = {1423131},
MRREVIEWER = {Jeremy\ T.\ Teitelbaum},
       DOI = {10.1007/978-3-642-61480-4},
       URL = {https://doi.org/10.1007/978-3-642-61480-4},
}

@article {hartl2017isogenies,
    AUTHOR = {Hartl, Urs},
     TITLE = {Isogenies of {A}belian {A}nderson {$A$}-modules and
              {$A$}-motives},
   JOURNAL = {Ann. Sc. Norm. Super. Pisa Cl. Sci. (5)},
  FJOURNAL = {Annali della Scuola Normale Superiore di Pisa. Classe di
              Scienze. Serie V},
    VOLUME = {19},
      YEAR = {2019},
    NUMBER = {4},
     PAGES = {1429--1470},
      ISSN = {0391-173X,2036-2145},
   MRCLASS = {11G09 (13A35 14K02 14L05)},
  MRNUMBER = {4050202},
MRREVIEWER = {Liang-Chung\ Hsia},
}

@incollection {hartl2020pinks,
    AUTHOR = {Hartl, Urs and Juschka, Ann-Kristin},
     TITLE = {Pink's theory of {H}odge structures and the {H}odge conjecture
              over function fields},
 BOOKTITLE = {{$t$}-motives: {H}odge structures, transcendence and other
              motivic aspects},
    SERIES = {EMS Ser. Congr. Rep.},
     PAGES = {31--182},
 PUBLISHER = {EMS Publ. House, Berlin},
      YEAR = {[2020] \copyright 2020},
      ISBN = {978-3-03719-198-9},
   MRCLASS = {11G09 (14D07)},
  MRNUMBER = {4321965},
}

@article {taelman2009special,
    AUTHOR = {Taelman, Lenny},
     TITLE = {Special {$L$}-values of {$t$}-motives: a conjecture},
   JOURNAL = {Int. Math. Res. Not. IMRN},
  FJOURNAL = {International Mathematics Research Notices. IMRN},
      YEAR = {2009},
    NUMBER = {16},
     PAGES = {2957--2977},
      ISSN = {1073-7928,1687-0247},
   MRCLASS = {11M38 (11G09)},
  MRNUMBER = {2533793},
MRREVIEWER = {Mihran\ Papikian},
       DOI = {10.1093/imrn/rnp038},
       URL = {https://doi.org/10.1093/imrn/rnp038},
}

@article {taelman2012special,
    AUTHOR = {Taelman, Lenny},
     TITLE = {Special {$L$}-values of {D}rinfeld modules},
   JOURNAL = {Ann. of Math. (2)},
  FJOURNAL = {Annals of Mathematics. Second Series},
    VOLUME = {175},
      YEAR = {2012},
    NUMBER = {1},
     PAGES = {369--391},
      ISSN = {0003-486X,1939-8980},
   MRCLASS = {11M38 (11G09)},
  MRNUMBER = {2874646},
MRREVIEWER = {Mihran\ Papikian},
       DOI = {10.4007/annals.2012.175.1.10},
       URL = {https://doi.org/10.4007/annals.2012.175.1.10},
}

@misc{Ye26,
      title={Artin twists of {D}rinfeld modules and {G}oss {L}-series}, 
      author={Jing Ye},
      year={2026},
      eprint={2602.04211},
      archivePrefix={arXiv},
      primaryClass={math.NT},
      note={arXiv: \url{https://arxiv.org/abs/2602.04211v2}},
}

@article {Pap08,
    AUTHOR = {Papanikolas, Matthew A.},
     TITLE = {Tannakian duality for {A}nderson-{D}rinfeld motives and
              algebraic independence of {C}arlitz logarithms},
   JOURNAL = {Invent. Math.},
  FJOURNAL = {Inventiones Mathematicae},
    VOLUME = {171},
      YEAR = {2008},
    NUMBER = {1},
     PAGES = {123--174},
      ISSN = {0020-9910,1432-1297},
   MRCLASS = {11J93 (11G09 12H10 14L17)},
  MRNUMBER = {2358057},
MRREVIEWER = {Liang-Chung\ Hsia},
       DOI = {10.1007/s00222-007-0073-y},
       URL = {https://doi.org/10.1007/s00222-007-0073-y},
}

@article {CP12,
    AUTHOR = {Chang, Chieh-Yu and Papanikolas, Matthew A.},
     TITLE = {Algebraic independence of periods and logarithms of {D}rinfeld
              modules},
      NOTE = {With an appendix by Brian Conrad},
   JOURNAL = {J. Amer. Math. Soc.},
  FJOURNAL = {Journal of the American Mathematical Society},
    VOLUME = {25},
      YEAR = {2012},
    NUMBER = {1},
     PAGES = {123--150},
      ISSN = {0894-0347,1088-6834},
   MRCLASS = {11J93 (11G09 11J89)},
  MRNUMBER = {2833480},
MRREVIEWER = {R.\ Wallisser},
       DOI = {10.1090/S0894-0347-2011-00714-5},
       URL = {https://doi.org/10.1090/S0894-0347-2011-00714-5},
}

@article {CM21,
    AUTHOR = {Chang, Chieh-Yu and Mishiba, Yoshinori},
     TITLE = {On a conjecture of {F}urusho over function fields},
   JOURNAL = {Invent. Math.},
  FJOURNAL = {Inventiones Mathematicae},
    VOLUME = {223},
      YEAR = {2021},
    NUMBER = {1},
     PAGES = {49--102},
      ISSN = {0020-9910,1432-1297},
       DOI = {10.1007/s00222-020-00988-1},
       URL = {https://doi.org/10.1007/s00222-020-00988-1},
}

@article {AT90,
    AUTHOR = {Anderson, Greg W. and Thakur, Dinesh S.},
     TITLE = {Tensor powers of the {C}arlitz module and zeta values},
   JOURNAL = {Ann. of Math. (2)},
  FJOURNAL = {Annals of Mathematics. Second Series},
    VOLUME = {132},
      YEAR = {1990},
    NUMBER = {1},
     PAGES = {159--191},
       DOI = {10.2307/1971503},
       URL = {https://doi.org/10.2307/1971503},
}

@article {Yu91,
    AUTHOR = {Yu, Jing},
     TITLE = {Transcendence and special zeta values in characteristic
              {$p$}},
   JOURNAL = {Ann. of Math. (2)},
  FJOURNAL = {Annals of Mathematics. Second Series},
    VOLUME = {134},
      YEAR = {1991},
    NUMBER = {1},
     PAGES = {1--23},
       DOI = {10.2307/2944331},
       URL = {https://doi.org/10.2307/2944331},
}

@article {GN24,
	AUTHOR = {Gezmi{\c{s}}, O{\u{g}}uz and Namoijam, Changningphaabi},
	TITLE = {On the transcendence of special values of {G}oss
	{$L$}-functions attached to {D}rinfeld modules},
	JOURNAL = {Int. J. Number Theory},
	FJOURNAL = {International Journal of Number Theory},
	VOLUME = {22},
	YEAR = {2026},
	NUMBER = {5},
	PAGES = {925--952},
	ISSN = {1793-0421,1793-7310},
	MRCLASS = {11G09 (11J93 11M38)},
	MRNUMBER = {5061511},
	DOI = {10.1142/S179304212650051X},
	URL = {https://doi.org/10.1142/S179304212650051X},
}

@article {GN25,
	AUTHOR = {Gezmi{\c{s}}, O{\u{g}}uz and Namoijam, Changningphaabi},
	TITLE = {On the algebraic independence of logarithms of {A}nderson
	{$t$}-modules},
	JOURNAL = {Kyushu J. Math.},
	FJOURNAL = {Kyushu Journal of Mathematics},
	VOLUME = {80},
	YEAR = {2026},
	NUMBER = {1},
	PAGES = {65--114},
	ISSN = {1340-6116,1883-2032},
	MRCLASS = {11J93 (11G09)},
	MRNUMBER = {5084608},
}

@article{ANDTR20,
	title = {On special {L}-values of $t$-modules},
	journal = {Advances in Mathematics},
	volume = {372},
	pages = {107313},
	year = {2020},
	issn = {0001-8708},
	doi = {https://doi.org/10.1016/j.aim.2020.107313},
	url = {https://www.sciencedirect.com/science/article/pii/S000187082030339X},
	author = {Angl\`es, Bruno and Ngo Dac, Tuan and Tavares Ribeiro, Floric}
}

@article {CCM22,
	AUTHOR = {Chang, Chieh-Yu and Chen, Yen-Tsung and Mishiba, Yoshinori},
	TITLE = {Algebra structure of multiple zeta values in positive
	characteristic},
	JOURNAL = {Camb. J. Math.},
	FJOURNAL = {Cambridge Journal of Mathematics},
	VOLUME = {10},
	YEAR = {2022},
	NUMBER = {4},
	PAGES = {743--783},
	ISSN = {2168-0930,2168-0949},
	MRCLASS = {11R59 (11J93 11M32)},
	MRNUMBER = {4524827},
	MRREVIEWER = {Vincent\ Bosser},
	DOI = {10.4310/cjm.2022.v10.n4.a1},
	URL = {https://doi.org/10.4310/cjm.2022.v10.n4.a1},
}

\end{document}